\documentclass[11pt]{amsart}

\usepackage[OT2, T1]{fontenc}
\usepackage{url}
\usepackage{amsmath}
\usepackage{stackengine}
\usepackage{array}
\usepackage{graphicx}
\usepackage{amsfonts}
\usepackage{amssymb}
\usepackage{amstext}
\usepackage{amsthm}
\usepackage{enumitem}
\usepackage{bm}
\usepackage{hyperref}
\usepackage{colonequals}
\usepackage{enumitem}
\usepackage[alphabetic,lite]{amsrefs}
\usepackage{cleveref}
\usepackage[all,cmtip]{xy}
\usepackage{fullpage}
\usepackage{appendix}
\usepackage{scalerel}
\usepackage{mathrsfs}
\usepackage{comment}
\usepackage{tikz-cd}
\usepackage{tikz}
\usepackage{amssymb} 
\def\acts{\curvearrowright}
\numberwithin{equation}{subsection}

\newtheorem{theorem}{Theorem}
\newtheorem{lemma}[theorem]{Lemma}
\newtheorem{proposition}[theorem]{Proposition}
\newtheorem{corollary}[theorem]{Corollary}
\newtheorem*{thm}{Theorem}
\newtheorem{conj}[theorem]{Conjecture}
\newtheorem{definition}{Definition}
\theoremstyle{definition}

\newtheorem{defn}[theorem]{Definition}
\newtheorem{notation}[theorem]{Notation}
\newtheorem{setup}[theorem]{Setup}
\newtheorem{ass}[theorem]{Assumption}

\theoremstyle{remark}
\newtheorem{remark}[theorem]{Remark}

\newtheorem*{claim}{Claim}
\newtheorem{example}[theorem]{Example}

\numberwithin{theorem}{section}

\newcommand{\bA}{\mathbb{A}}

\newcommand{\bC}{\mathbb{C}}
\newcommand{\bD}{\mathbb{D}}

\newcommand{\bF}{\mathbb{F}}
\newcommand{\bG}{\mathbb{G}}
\newcommand{\bH}{\mathbb{H}}

\newcommand{\bL}{\mathbb{L}}

\newcommand{\bN}{\mathbb{N}}

\newcommand{\bQ}{\mathbb{Q}}
\newcommand{\bR}{\mathbb{R}}

\newcommand{\bZ}{\mathbb{Z}}
\newcommand{\Qpun}{\bQ_p^{\textrm{un}}}

\newcommand{\bbJ}{\mathbf{J}}

\newcommand{\cA}{\mathscr{A}}
\newcommand{\cB}{\mathcal{B}}
\newcommand{\cC}{\mathcal{C}}
\newcommand{\cD}{\mathcal{D}}

\newcommand{\cL}{\mathcal{L}}
\newcommand{\cM}{\mathcal{M}}

\newcommand{\cO}{\mathcal{O}}
\newcommand{\cP}{\mathcal{P}}
\newcommand{\cQ}{\mathcal{Q}}

\newcommand{\cS}{\mathcal{S}}
\newcommand{\cT}{\mathscr{T}}

\newcommand{\cZ}{\mathcal{Z}}
\newcommand{\into}{\hookrightarrow}

\newcommand{\ShimK}{S_K(G,\cD)}
\newcommand{\integralShimK}{\mathcal{S}_K(G,\cD)}
\newcommand{\ShimKprime}{S_{K_H}(H,\cD')}
\newcommand{\integralShimKprime}{\mathcal{S}_{K_H}(H,\cD')}

\newcommand{\fp}{\mathfrak{p}}

\newcommand{\fG}{\mathfrak{G}}

\newcommand{\et}{\text{\'et}}
\newcommand{\dR}{{\mathrm{dR}}}
\newcommand{\cris}{{\mathrm{cris}}}

\newcommand{\pdiv}{\mathscr{G}}
\newcommand{\Ag}{\cA_g}

\newcommand{\la}{\langle}
\newcommand{\ra}{\rangle}

\newcommand{\Fpbar}{\overline{\mathbb{F}}_p}
\DeclareMathOperator{\GL}{GL}

\DeclareMathOperator{\GSp}{GSp}

\DeclareMathOperator{\GSpin}{GSpin}

\DeclareMathOperator{\Spf}{Spf}

\DeclareMathOperator{\Gal}{Gal}
\DeclareMathOperator{\End}{End}
\DeclareMathOperator{\Hom}{Hom}

\DeclareMathOperator{\Lie}{Lie}
\DeclareMathOperator{\Res}{Res}
\DeclareMathOperator{\Spec}{Spec}
\DeclareMathOperator{\Ext}{Ext}
\DeclareMathOperator{\ad}{ad}
\DeclareMathOperator{\Span}{Span}

\DeclareMathOperator{\Fil}{Fil}

\DeclareMathOperator{\diag}{diag}

\DeclareMathOperator{\bfx}{\mathbf{x}}
\DeclareMathOperator{\bfy}{\mathbf{y}}

\DeclareMathOperator{\coker}{coker}

\DeclareMathOperator{\rk}{rk}
\DeclareMathOperator{\tor}{tor}

\DeclareMathOperator{\Frob}{Frob}
\DeclareMathOperator{\gr}{gr}
\DeclareMathOperator{\im}{im}

\DeclareMathOperator{\loc}{loc}

\DeclareMathOperator{\Cl}{Cl}

\DeclareMathOperator{\mult}{mult}

\DeclareMathOperator{\MTT}{MT}

\DeclareMathOperator{\sep}{sep}

\begin{document}
\author{Ruofan Jiang and Ananth N. Shankar}

\title{$S$-integrality for families of ordinary K3 surfaces and algebraicity theorems}
\begin{abstract}
    We prove an $S$-integrality theorem for special divisors on GSpin Shimura varieties in positive characteristic. Let $C$ be a generically ordinary curve in such a Shimura variety not contained in any special divisor. Then, for any increasing sequence of prime-to-$p$ positive integers $m_i$ and any finite set of closed points $S\subset C$, we prove that $C\setminus S$ meets the special divisor $Z(m_i)$ for all but finitely many $i$. 

The key new input is an algebraicity theorem for formal special endomorphisms which allows us to use techniques from Diophantine approximation. We prove this algebraicity theorem using a punctual monodromy theorem and a positive-characteristic analogue of the Mumford--Tate conjecture.
\end{abstract}
\maketitle
\tableofcontents

\section{Introduction}
\subsection{Main results}
In this paper, we prove the following Diophantine theorem about families of K3 surfaces in characteristic $p > 2$. The geometric setup is that of a generically ordinary non-isotrivial family of K3 surfaces $\mathscr{X}\to C$ where $C/\Fpbar$ is a smooth curve. Let $\eta$ be the generic point of $C$. Given any point $P$ of $C$, let $\textrm{Pic}_P$ denote the Picard lattice of $\mathscr{X}_P$. Our standing assumption will be that the discriminant of $\textrm{Pic}_\eta$ is prime to $p$, where $\eta$ is the generic point of $C$. Given any point $x\in C$, there is a natural injective specialization map $\textrm{Pic}_\eta \to \textrm{Pic}_x$. Given this setup, our main theorem is as follows. 

\begin{theorem}\label{thm:main1}
Let $m_i>0$ be any increasing sequence of prime-to-$p$ integers, and let $S\subset C$ denote any finite set of closed points. For all but finitely many $i$, there are points $x_i \in C\setminus S$ such that there is a class $\alpha_i\in \textrm{Pic}_{x_i}$ that is orthogonal to $\textrm{Pic}_\eta \subset \textrm{Pic}_{x_i}$ such that $\alpha_i^2 = -m_i$.
\end{theorem}

The above result follows from the more general result stated in the setting of Shimura varieties associated to Spinor groups. To that end, we set $\cM/\Fpbar$ to denote the special fiber of a GSpin Shimura variety with good reduction at $p$. For every integer $m$, there is a divisor $Z(m) \subset \cM$, which is the special fiber of a divisor which is also a GSpin Shimura variety in its own right. Divisors of this sort are called \emph{special divisors}. We will define all these notions precisely in Section \ref{sec: prelim}. 

This more general setup specializes to the one above as polarized K3 surfaces are parameterized by such Shimura varieties, and the locus of points where the parameterized K3 surfaces admit extra Picard classes is parameterized by special divisors of the form $Z(m)$. We now state our theorem:

\begin{theorem}\label{mainShimura}
  Let $C$ be a smooth connected (affine) curve over $\Fpbar$ with 
a non-constant morphism $C\rightarrow \cM$ whose image intersects the ordinary locus, and such that $C$ does not map to any special divisor $Z(m)$. Let $m_i$ be an increasing sequence of prime-to-$p$ integers and let $S \subset C$ be any finite set of closed points. For all but finitely many $i$, the intersection $ (C\setminus S) \cap Z(m_i)$ is non-empty.
\end{theorem}

We note that we may choose the sequence $m_i$ and the finite set $S\subset C$ independently.  Our theorem is motivated by the following Diophantine question in characteristic 0. Let $\cA_1$ denote the modular curve, and let $x\in \cA_1(\overline{\bQ})$ be some point. Then, it is conjectured that given any finite set of primes $S$, there are only finitely many CM points $x_i$ that are $S$-integral in $\cA_1 \setminus \{ x\}$. Equivalently, there are only finitely many CM points $x_i$ such that $x$ is $S$-integral in $\cA_1 \setminus \{x_i\}$. This conjecture was motivated in turn by work of Baker-Ih-Rumely, who in \cite{BIR} prove the analogue of this with $\cA_1$ replaced by $\bG_m$ or an elliptic curve $E$, with some non-torsion point taking the role of $x$ and torsion points taking the role of CM points. Indeed, our theorem has the obvious reformulation. 

\begin{theorem}\label{thm: Sintegralformulation}
    Let $K$ be a global function field in characetistic $p$ and let $x \in \cM(K) \setminus \cM(\Fpbar)$ be an ordinary point. Suppose that $x$ is not contained in any $Z(m)$. Let $S$ be any finite set of places of $K$. Then there are only finitely many prime-to-$p$ integers $m$ such that $x$ extends to an $S$-integral point of $\cM_K \setminus Z(m)_K$. 
\end{theorem}

\begin{remark}
    We note that the a product of two modular curves, and the moduli space of principally polarized abelian surfaces also falls under this regime. In the first case, special divisors are the loci $Y_0(m)$ which parameterize pairs of elliptic curves with a cyclic $m$-isogeny, and in the second case the special divisors $Z(m)$ are the loci of points with endomorphisms by the order $\bZ[\sqrt{m}]$. We note that the work \cite{JJ23} proves this theorem in the setting of $\cA_2$, but with a non-ordinariness hypothesis. 
\end{remark}

\begin{remark}
    We note that the notation $\cM$ for our Spin Shimura varieties is only for the introduction. We will use the notation $S$ for these varieties, starting from Section \ref{sec: prelim} onwards. 
\end{remark}
\begin{remark}
    Upon finishing the paper, we have heard that Robin Huang has proved a special case of Theorem~\ref{thm:main1} where the collection $\{m_i\}$ are square classes (i.e., of the form $\{Dm^2\}$ for some fixed $D$) via a density estimation. See \cite{RH} for more details.  
\end{remark}

\subsubsection{Picard-rank jumps}
We now work with the setup of Theorem \ref{mainShimura}. In previous work, specifically \cite{MAT}, \cite{MST22} and \cite{JSY26}, the authors prove that intersection $ C \cap \bigcup_m Z(m)$ is infinite, but \emph{with the intersection being over all positive integers $m$}. This is strong enough to imply infinitely many Picard-rank-jumping specializations in families of K3 surfaces. However, the methods in those works are not able to address the stronger question of arbitrarily thin sequences of positive integers $m$. Theorem \ref{mainShimura} has the following geometric application, which requires that we work with a thin (density zero) sequence of $m$. 

\subsubsection{Application to cubic fourfolds}
We refer to work of Hassett \cite{hassett} for the following. A cubic fourfold is said to be \textbf{special}, if it contains a surface not homologous to a complete intersection. Special cubic fourfolds form a countable union of special divisors in the moduli space of cubic fourfolds, which admits an open immersion to a suitable GSpin Shimura variety of signature $(2,20)$. Among the countable union, there is an infinite but thin subfamily of divisors that are simultaneously moduli spaces of polarized K3 surfaces (with varying degrees of polarization). See \cite[Theorem 1.0.2]{hassett} for the exact definition of $Z(m)$ that are allowed. We call a special cubic fourfold lying on such a divisor \textbf{having an associated K3 surface}. When $p> 3$, the GSpin Shimura variety in question has good reduction at $p$, and all definitions made above make sense over char $p$ as well. Theorem \ref{mainShimura} applies to obtain the following consequence.
\begin{theorem}Let $\mathcal{Y}\rightarrow C$ be a non-isotrivial family of ordinary cubic fourfold in characteristic $p>3$ with coprime to $p$ generic discriminant. Then there are infinitely many points on $C$ over which the fiber of $\mathcal{Y}$ has an associated K3 surface.
\end{theorem}

\subsection{Method of proof and a group-theoretic Mumford-Tate theorem}
We will now explain the difficulties that previously obstructed the passage from all $m$ to infinitely many $m$, and how we resolve these obstructions. The proof has four main ingredients:
\begin{enumerate}
    \item a local-global comparison of intersections with special divisors;
    \item $\overline{\bQ}$-algebraicity of formal special endomorphisms and Diophantine approximation;
    \item a punctual monodromy theorem and algebraization of formal special endomorphisms;
    \item the Basic Mumford--Tate theorem.
\end{enumerate}
The first ingredient builds on methods developed in previous work, while the latter three are new to this paper. We now explain how these ingredients fit together. As in \cite{CO06}, \cite{Charles}, \cite{ST20}, \cite{SSTT}, \cite{MAT}, \cite{MST22} and \cite{JSY26}, we use a local-global method to attack this question. The broad idea is to start with a proper curve $C$ that maps to a compactification of $\cM$, estimate the global intersection number of $C$ with $Z(m)$ using Borcherds theory, and prove that the local intersection number $i_P(C.Z(m))$ at any fixed point $P$ is ``small'' relative to the global intersection number. Supposing that there was an infinite sequence of $Z(m)$ that intersects $C$ at only a fixed finite set of points of $C$, then the local intersection numbers would add up to something smaller than the global intersection numbers, a contradiction. 

\subsubsection{\textbf{Formal special endomorphisms}}
The difficulty, of course, is to effectively bound $i_P(C.Z(m))$ in terms of $m$. Here, the moduli theory of $Z(m)$ is strongly used. We pick a local parameter $t$ at $P$, and consider the complete local ring $\Fpbar[[t]]$ of $C$ at $P$. Then, there is a decreasing nested sequence of lattices $L_n$ of ``special endomorphisms'' equipped with a quadratic form $Q$ (compatible with inclusions) associated to $\Fpbar[[t]]/(t^n)$. The lattices $L_n$ encode the local intersection number $i_P(C. Z(m)) = \sum_{n=1}^{\infty}\#\{v\in L_n : Q(v) = m\}$. In the complex setting of this question, there exists an integer $n_0$ such that $L_n = \{0\}$ for $n\geq n_0$. In our setting this is never true, and indeed the ranks of the $L_n$ are independent of $n$. The main difficulty is the existence of so-called ``\textbf{formal special endomorphisms}'' over $\Fpbar[[t]]$. This term will be defined in Section \ref{sec: prelim}, but this phenomenon is the hardest obstruction to overcome, and arises as follows: despite $\bigcap_n L_n = \{0\}$, it is possible that $\bigcap_n (L_n\otimes \bZ_p) = \Lambda \neq \{0\}$. This phenomenon could have the a priori consequence that there could be a vector $v\in L_n$ with $Q(v) = m$, where $n$ is arbitrarily large relative to $m$. Here is an example that illustrates this. 

\begin{example}
    Suppose that $L_1=\bZ e\oplus \bZ f$ is of rank 2, and the sublattice $\Lambda\subseteq L_1\otimes\bZ_p$ is generated by $\lambda=f+\sum_{i=0}^\infty p^{i!} e$. We get a sequence $\lambda_r:=f+\sum_{i=0}^r p^{i!} e\in L$ converging to $\lambda$ quickly; indeed,  $\lambda_r-\lambda \in p^{(r+1)!}L_1\otimes\bZ_p$. Using local computations, it is easy to show that $\lambda_r \in L_{p^{(r+1)!}}$. Now $p^{(r+1)!}$ is a huge number compared to $\la \lambda_r,\lambda_r\ra$. This shows the existence of special endomorphism $v\in L_n$ where $n$ is extremely large relative to $Q(v)$.
\end{example}

We note that such an example can certainly arise in the setting of a formal curve mapping to $\cM$ (see \cite[Section 3]{MST22}), but shouldn't arise in the setting of an algebraic curve $C\to \cM$ because $i_P(C.Z(m))$ has the obvious and extremely coarse bound given by the global intersection number $(C.Z(m))$. In \cite{SSTT}, \cite{MST22} and \cite{JSY26}, the authors are able to use this coarse bound, which is of course tautologically insufficient to obtain effective control of $i_P(m)$ for a single value of $m$, to obtain efficient control of $i_P(m)$ on average over $m$, i.e. effective control of $\sum_m i_P(m)$. By its very nature, such an argument cannot work to bound $i_P(m)$ for any single value of $m$, and so we need new ideas in our work. 

\subsection{Algebraicity of formal special endomorphisms}
The obstruction described above arises precisely when $\Lambda \subset L_1$, the lattice of formal special endomorphisms, is $p$-adically well approximated by rational subspaces of $L_1$. One of our new ideas is to prove the following result. 

\begin{thm}[Theorem \ref{thm: Qbar algebraicity of formal special endomorphisms}]\label{thm:introQbar}
    Notation as above. The subspace $\Lambda \otimes \bQ_p \subset L_1\otimes \bQ_p$ is is defined over a number field $F\subset \bQ_p$.
\end{thm}
This result is very surprising, as there is an identical statement that could be made in the number field setup, which we strongly believe to be false! The setup is as follows. Let $E/K$ be an elliptic curve over a number field, let $v$ be a place of $K$ of good supersingular reduction for $E$ and such that $K_v$ is the quadratic unramified extension of $\bQ_p$. It is well known that this forces $E_{K_v}$ to have ``formal CM'', i.e. the $p$-divisible group associated to $E_{K_v}$ admits extra endomorphisms, necessarily by $\cO_{K_v}$. Let $E_0$ denote the mod $v$ reduction of $E$. This yields a tautological injective map $\cO_{K_v} \to \End(E_0)\otimes \bZ_p$. The statement analogous to Theorem \ref{thm:introQbar} would be that the subspace of $\End(E_0)\otimes \bQ_p$ defined by the image of $\cO_{K_v}\otimes \bQ_p$ actually descends to a number field $F\subset \bQ_p$. We expect this to be false. 

Given Theorem \ref{thm:introQbar}, we are able to effectively bound $i_P(C.Z(m))$ by using a $p$-adic variant of Schmidt's subspace theorem, which basically says that an algebraic irrational subspace of a $\bQ$-vector space does not admit good $p$-adic approximations by rational subspaces. We prove our algebraicity theorem by proving an algebraization theorem for formal special endomorphisms, and by formulating and proving a ``Basic Mumford--Tate conjecture'' - a positive characteristic analogue of the Mumford--Tate conjecture. 

\subsection{Punctual monodromy and algebraization of formal special endomorphisms}
We note that our Shimura variety is equipped with a map to $\cA_g$, the moduli space of principally polarized abelian varieties. At any point $y\in C$, special endomorphisms are endomorphisms of the universal abelian scheme $\mathscr{A}$ pulled back to $x$, and formal special endomorphisms are endomorphisms of the universal $p$-divisible group $\mathscr{A}[p^{\infty}]$ pulled back to $y$. Applying this to our setup of the formal neighbourhood $\Fpbar[\![t]\!]$ of $C$ at $x$, an element $\lambda\in \Lambda$ is an endomorphism of $\mathscr{A}[p^\infty]|_{\Fpbar[\![t]\!]}$. Then the punctual monodromy theorem roughly says that the purely formal data $\Lambda$ encodes information of global motivic cycles on $\mathscr{A}_C$: 
\begin{thm}[Theorem \ref{theorem: algebraic}]
    A full rank sublattice of $\Lambda$ propagates globally to endomorphisms of $\mathscr{A}[p^\infty]$ over $C$, and $C$ is contained in the special fiber of a special subvariety of codimension $\rk \Lambda$. 
\end{thm}
In particular, the larger the lattice $\Lambda$, the more special the curve $C$. To prove this, we make crucial use of \cite{Ruofan}, work of Chai \cite{Ch03}, and the parabolicity conjecture proved by \cite{MD20}. A special case of this, where $P$ is ordinary, essentially follows from the mod $p$ log Ax--Lindemann conjecture for GSpin Shimura varieties, which is solved in \cite{Ruofan}. When $P$ is not ordinary, the problem becomes more subtle, since we don't have a theory of canonical coordinates for $\mathscr{S}^{/P}$. However, in our situation, we are able to circumvent the difficulty by taking an ``ordinary formal loop around $P$'' (which amounts to $\Spec \mathbb{F}(\!(t)\!)$) and use the idea of ``relative Serre--Tate coordinates''; cf. \cite{Ch03}. The process of taking a ``formal loop'' is the reason for the adjective ``\textit{punctual}''. 

It is also worth noting that in the previous works, we have proved analogues of this theorem when $C$ is not generically ordinary (see \cite{JJ23,JSY26}) under the much stronger assumption that  $\rk\Lambda$ equals the rank of the slope 0 part of the generic K3 crystal over $C$. 

\subsection{Basic Mumford--Tate}
In this section, we formulate a purely group theoretic positive characteristic analogue of the Mumford--Tate conjecture. The setting will be of subvarieties of Shimura varieties that intersect the basic locus. We first fix a prime $p$. We will let $\Ag$ denote the mod $p$ special fiber of the moduli space of principally polarized abelian schemes. When we use the term ``mod $p$ Shimura variety'', we implicitly fix we implicitly fix a prime $\frak{p} \mid p$ of the reflex field of a Shimura variety $S$, an integral model of the associated Shimura variety at $\mathfrak{p}$, and we let $S_{\fp}$ denote the mod $\frak{p}$ special fiber of this integral model. We say that $S$ is a Shimura subvariety with \emph{good reduction} if the integral model at $\frak{p}$ is an integral canonical model as per \cite{KM09} or \cite{BST}. In the setting of Hodge type, this is equivalent (by \cite{KM09}) to the level structure at $p$ being hyperspecial. When the context is clear, we will drop the subscript of $\fp$ and will let $S$ also refer to the mod $\fp$ Shimura variety.

\subsubsection{Supersingular points} We work in the setting of $\GSp(V)$, i.e. in the setting of subvarieties of $\Ag$ which intersect the supersingular locus. We let $V_{\et,\ell}$ denote the $\ell$-adic local system on $\Ag$ induced by $V$. Note that this is just the relative etale cohomology of the universal abelian scheme. Similarly, let $_{\cris}V$ denote the $F$-isocrystal induced by $V$. 

Let $X$ be a smooth connected variety that maps to $\cA_g$ and suppose that $x\in X$ maps to a supersingular point. We consider $V_{\et,\ell}|_X$. Taking the fiber at $x$, we obtain a representation
\[
\pi_{1,\et}(X,x)\acts V_{\et,\ell,x},
\]
and the induced monodromy representation
\[
\pi_{1,\et}(X,x)\longrightarrow \End(V_{\et,\ell,x}).
\]
Let $H_\ell$ denote the connected component of the monodromy image. $\End^0(A_x)$ admits a natural $\bQ_\ell$ structure
\[
\End^0(A_x)\subset \End^0_{\bQ_\ell}(V_{\et,\ell}).
\]
A similar crystalline construction can be made, with the Tannakian monodromy group of the Isocrystal $V_{\cris}|_X$ taking the place of $V_{\et,\ell}$ and the $\Qpun$-group $H_{p}$ in place of $H_{\ell}$.

Let $H$ denote the smallest subgroup of $\GL(\End^0(A_x))$ such that
\[
H_{\bQ_\ell}\supset H_\ell \ \forall \ell\neq p, H_{\Qpun} \supset H_p 
\]

We make the following conjecture
\begin{conj}[Supersingular Mumford-Tate]\label{conj: ss MT}
    We have $H_{\bQ_\ell} = H_\ell$, $H_{\Qpun} = H_p$.
\end{conj}

We prove the following case of our conjecture.
\begin{theorem}\label{thm: MT GSpin}
Suppose that $X$ is generically ordinary and its image is contained in a Shimura subvariety $S\subset \Ag$, where the Shimura datum defining $S$ is a $\GSpin$ group. Then Conjecture \ref{conj: ss MT} is true for $X$.
\end{theorem}
There is also a formulation of this conjecture when the mod $p$ Shimura variety does not intersect the supersingular locus of $\cA_g$. We postpone the formulation of this setting to Section \ref{sec: Mumford Tate}. The main ingredient in our proof is the work in \cite{Ruofan}, which proves that the monodromy $H_\ell$ is determined by the smallest Shimura subvariety containing $X$. 

\subsection*{Organization of the paper}
In Section 2, we define the various Shimura varieties in play. In Section 3, we recall the intersection-theoretic setup and reduce our main theorem to establishing local bounds. In Section 4, we prove the Punctual Monodromy Theorem. In section 5, we prove the supersingular Mumford-Tate conjecture in the setting of orthogonal Shimura varieties and use this to prove that the subspace of formal special endomorphisms at a point descend to an algebraic extension of $\overline{\bQ}$. In Section 6, we recall the local structure of a GSpin Shimura variety at closed points. In Section 7 and Section 8 we prove the key decay results over interior and boundary, respectively. In Section 9 we feed the algebraicity result proved in section 5 and the decay results proved in Section 7 and 8 to the mechinery of Diophantine approximation and obtain the desired control on the local intersection. We assemble all of these and finish the proof of our main theorem in Section 10.


\subsection*{Acknowledgements}
We are very grateful to Arul Shankar and Yunqing Tang for several valuable conversations on the $S$-integrality conjecture. The second author thanks Jacob Tsimerman for informing him about the results of Baker-Ih-Rumely and for suggesting that a version of that result in the setting of the modular curve and CM points should hold. We are also grateful to Keerthi Madapusi, Davesh Maulik, Salim Tayou and Ziquan Yang for helpful conversations. 
AS was supported by the NSF grant DMS-2338942 and a Sloan research fellowship.

\section{Preliminaries}\label{sec: prelim}
\subsection{GSpin Shimura varieties and their integral canonical models} \label{subsec: set up SV}
We will set up basic definitions and terminology for our Shimura varieties, special endomorphisms, and special divisors. Our exposition and setup will follow \cite[Section 2]{MST22} and \cite[Section2]{SSTT}, and we will refer to \cite{AGHMP18} and \cite{MP16}. 

Let $L$ be a quadratic lattice with signature $(b,2)$ which is self-dual at $p$. Write $V$ for $L_\bQ$. We may also assume that $L \subset V$ is maximal among those lattices contained in $V$ on which the bilinear form has integer values. Let $\Cl(L)$ denote the Clifford algebra associated to $L$. Note that we have a natural embedding of free $\bZ_{(p)}$-modules $L\otimes \bZ_{(p)} \hookrightarrow \Cl(L)\otimes \bZ_{(p)}$. Let $G:= \mathrm{GSpin}(L\otimes \bZ_{(p)})$ be the group of spinor similitudes. Note that $G$ is a reductive group over $\bZ_{(p)}$ and is naturally a subgroup of $\Cl(L \otimes \bZ_{(p)})^{\times}$.  Let $\cD $ be the Hermitian symmetric domain $\{ z \in V_\bC | \langle z, \overline{z} \rangle < 0 \}$. Then $(G, \cD)$ defines a spinor Shimura datum with reflex field $\bQ$. For every neat compact open subgroup $K \subseteq G(\bA_f)$, we obtain a Shimura variety $\ShimK$ over $\bQ$. 

Set $H = \mathrm{Cl}(L)$, viewed as a $\mathrm{Cl}(L)$-bimodule. Note that $H$ has a natural $\bZ/2 \bZ$-grading. Left multiplication of $\Cl(L)$ on $H$  induces spin representation $G \to \mathrm{GL}(H)$ and an embedding $L \into \Cl(L) \into \End (H)$. Equip $\End (H)$ with the pairing $(\alpha, \beta) := 2^{-\mathrm{rank\,} L} \mathrm{tr}(\alpha \circ \beta)$. Then the composite embedding $L \into \End(H)$ is isometric and we can form the orthogonal projection $\bm{\pi} : \End(H) \to L$. The spinor group $G$ can be conversely viewed as the stabilizer of the $\bZ / 2 \bZ$-grading, the right $\Cl(L\otimes \bZ_{(p)})$-action and the idempotent projector $\bm{\pi}$.

Suppose now that $K$ is of the form $K_p K^p$ for $K^p \subseteq G(\bA^p_f)$ and $K_p = G(\bZ_p)$, i.e., is hyperspecial at $p$. 
By \cite{KM09} (cf. \cite{MP16}) there is an canonical integral model $\integralShimK$ over $\bZ_{(p)}$. 
We may endow $H$ with a suitable symplectic form such that the left multiplication by $G$ on $H$ respects the form up to scaling, and indeed induces an embedding of Shimura data $(G, \cD) \into (\mathrm{GSp}, \mathcal{H}^\pm)$, where $\mathcal{H}^\pm$ is the associated Siegel half spaces. This Siegel embedding equips $\integralShimK$ with a universal abelian scheme $\cA$, called the \textit{Kuga-Satake abelian scheme}.\footnote{Technically, in \cite{KM09} and \cite{MP16}, $\cA$ is only defined as a sheaf of abelian schemes up to prime-to-$p$ quasi-isogeny. 
However, for $K^p$ sufficiently small, we can take $\cA$ to be an actual abelian scheme (cf. \cite[(2.1.5)]{KM09}).}  We use $\cA[p^\infty]$ to denote the $p$-divisible group associated to $\cA$. Throughout, we fix a choice of such level structure, and for brevity let $\cS$ denote the integral model $\integralShimK$. We denote the special fiber $\cS_{\bF_p}$ by $S$.

Define the sheaves $\bH_\mathrm{B}/\cS_{\bC}$, $\bH_\ell/\cS $ ($\ell \neq p$), $\bH_p/\ShimK$, $\bH_\dR/\cS$ and $\bH_\cris /S$ denote the (first relative) Betti cohomology, $\ell$-adic etale cohomology, $p$-adic etale cohomology, de Rham cohomology, and the crystalline cohomology of the universal abelian scheme. 
The abelian scheme $\cA$ is equipped with a ``CSpin-structure'': a $\bZ / 2 \bZ$-grading, $\mathrm{Cl}(L)$-action and an idempotent projector $\bm{\pi}_? : \End(\bH_?) \to \End(\bH_?)$ for $?\in \{B, \ell, p, \dR, \cris\}$ on (various applicable fibers of) $\cS$. We use $\bL_?$ to denote the images of $\bm{\pi}_?$. 


\begin{definition}\label{def: special end} We now use the sheaves $\bL_?$ to define the notions of special endomorphisms of points of $\cS$ and special divisors in $\cS$.
    \begin{enumerate}
        \item Given any $\cS$-scheme $T$, an endomorphism $f \in \End(\cA_T)$ is called a \emph{special endomorphism} (\cite[Def.~5.2, see also Lem.~5.4, Cor.~5.22]{MP16}) if all cohomological realizations of $f$ lie in the image $\bL_? \rightarrow \End(\bH_?)$ where $? = \mathrm{B}, \dR, \cris, \ell \neq p, p$. For brevity, we simply drop those subscripts that do not make sense.  For example, if $p$ is invertible in $T$, then $? = \cris$ doesn't make sense and if $p$ is not invertible then $? = p,\mathrm{B}$ do not make sense. 
    
    We denote the submodule of $\End(\cA_T)$ of special endomorphisms by $L(\cA_T)$. By \cite[Lem. 5.2]{MP16}, for $v\in L(\cA_T)$, we have $v\circ v=[Q(v)]$ for
some $Q(v) \in \bZ_{\geq 0}$,  and $Q(v)$ is a positive definite quadratic form on the $\bZ$-lattice $L(\cA_T)$.
\item For $m\in \bZ_{>0}$, the \textit{special divisor} $\cZ(m)$ is the Deligne–Mumford stack over $\cS$ with
functor of points $\cZ(m)(T) = \{v\in L(\cA_T)\, |\,Q(v) = m\}$ for any $\cS$-scheme $T$. We use the
same notation for the image of $\cZ(m)$ in $\cS$. By  \cite[Prop. 4.5.8]{AGHMP18}, $\cZ(m)$ is an
effective Cartier divisor flat over $\bZ_{(p)}$ and hence $\cZ(m)_{\bF_p}$
is still an effective Cartier divisor of
$S$. We denote $\cZ(m)_{\bF_p}$ by $Z(m)$.

    \item Suppose that $p$ is not invertible in $T$. We say that $f\in \End(\cA_T[p^{\infty}])$ is a \emph{formal special endomorphism} if its crystalline realization lies in $\bL_{\cris}$. We denote the $\bZ_p$-submodule of $\End(\cA_T[p^{\infty}])$ of formal special endomorphisms by $\cL(\cA_T)$.
\end{enumerate}
     \end{definition}

\begin{remark}
    Consider a connected scheme $T$ and a $T$-valued point of $\cS$. An endomorphism $f\in \End(\cA_T)$ or $\End(\cA_T[p^{\infty}])$ being special can be checked at any geometric point of $T$ (\cite[Proposition 4.3.4]{AGHMP18}). Further if $x\in T_{\bF_p}$ is a geometric point, then $f$ is special if and only if $f_{t,\cris} \in \bL_{\cris}$. 
\end{remark}

\subsection{Shimura subvarieties of $\ShimK$.}\label{subsec: subshimura varieties}
Work of Fiori \cite[Theorem 3.6]{F18} has a clean and complete classification of sub-Shimura varieties of $\ShimK$. Roughly speaking, every Shimura subvariety of $\ShimK$ is induced by a morphism of Shimura daya $(H,\cD') \to (G,\cD)$, where $(H,\cD')$ (upto modifying the center of $H$) falls in one of the following two cases. 
\begin{enumerate}
    \item $H$ is the restriction of scalars from a totally real field of an orthogonal group that is compact at all but a single place, and has signature $(b',2)$ at the last place. We note that this includes every \emph{Special Divisor}, where the totally real field is $\bQ$, and we have $b' = b-1$. 
    \item $H$ is the restriction of scalars of a unitary group defined over a totally real field. The unitary group is defined by a Hermitian form having signature $(b',1)$ at a single real place, and is positive definite at every other real place. 
\end{enumerate}

For a level $K_H$ compatible with $K$, we denote the corresponding Shimura subvariety by $\ShimKprime$.
In both cases, there is a number field $F$ that acts on $L$ and therefore on $\bL_? \otimes \bQ$. In the first case, $F$ is the totally real field mentioned above. In the second case, $F$ is a totally imaginary (and therefore CM) extension of the totally real field mentioned above. In both cases, $F$ will be the reflex field of the Shimura datum $(H,\cD')$. Suppose that $\ShimKprime$ is not contained inside any special divisor as defined in the previous subsection. Then, the action of $F$ on $L$ allows us to view $L$ as an $F$-vector space. Further, given any complex point $x \in \ShimKprime$, the module of special endomorphisms $L_x$ at $x$ will be an $F$-subspace of $\bL_{B,x} \otimes \bQ$. Specifically, when restricted to $\ShimKprime$, there is an $F$-vector space $W$ and a canonical isomorphism $L \simeq \textrm{Res}_{F/\bQ} W$ and the $F$-action on $\bL_?$ is induced by this isomorphism. Further, $L_x \subset \bL_{B,x}\otimes \mathbb{Q}$ has the form $\textrm{Res}_{F/\bQ} W'$ where $W'\subset W$ is an $F$-subspace.

It is known (for example, by work of Mayinski \cite{mayinski}) that all these Shimura subvarieties are relatively PEL inside $\ShimK$, i.e. each Shimura subvariety of $\ShimK$ is defined by the locus in $\ShimK$ of points at which the Kuga-Satake abelian variety has specific extra endomorphisms.

Now, for such a Shimura subvariety $\ShimKprime \subset \ShimK$, its reflex field $F$, we define an integral model $\integralShimKprime$, denoted by $\cS'$, by taking the closure of $\ShimKprime$ in $\cS_{\mathcal{O}_{F}}$. 
Consider a prime $p>2$ as above, and a place $\mathfrak{p}$ of $F$ dividing $p$. We denote the special fiber $\cS'_{\mathcal{O}_F/\mathfrak{p}}$ by $S'$. By the main theorem of \cite{Noot}, we have that the formal completion of $\cS'$ at any closed ordinary point $x\in\cS'$ is a finite union of formal subtorii of the formal completion of $\cS$ at $x$. Here, we say that $x$ is ordinary if it is ordinary when considered as a point of $\cS$. 

The above discussion apply verbatim to special subvarieties of $\ShimK$; by definition, a special subvariety is a Hecke translation of a Shimura subvariety; cf. \cite{BMI98}.  One of the main results of \cite{Ruofan} is that any minimal special subvariety whose integral model contains an irreducible generically ordinary subvariety $X_0 \subseteq \cS_{\bF_q}$ is determined by the $\ell$-adic monodromy of $\bH_{\ell}|_{X_0}$ (or the overconvergent $p$-adic monodromy of $\bH_\cris|_{X_0}$): 
\begin{theorem}
    Suppose $X_0$ is a geometric connected smooth 
 variety over $\mathbb{F}_q$ with a morphism $f_0$ into $\cS_{\mathbb{F}_q}$, whose image lies generically in the ordinary locus. Define $G_\mathrm{B}(f_0)$ as the generic Mumford--Tate group of $\bH_{\mathrm{B}}$ over a minimal special subvariety whose integral model contains the image of $f_0$\footnote{Minimal special subvarieties having this property are Hecke translates of each other and share the same generic Mumford--Tate group.}. Then $G_\mathrm{B}(f_0)\otimes \bQ_l$ (resp. $G_\mathrm{B}(f_0)\otimes \bQ_p$) identifies with the neutral component of the $l$-adic monodromy group of $\bH_l|_{X_0}$ 
 (resp. $p$-adic monodromy group of $\bH_{\cris}|_{X_0}$). 
\end{theorem}

\subsection{Toroidal compactifications}\label{sub:toroidalSpin}
We very briefly review toroidal compactification and their integral models for GSpin Shimura varieties. This is just a special case of a more general theory (cf. \cite{MP19}), which we don't intend to recall.

Let $L$ be a quadratic $\bZ$-lattice of signature $(2,b)$ which is self-dual at $p$. As explained earlier, the GSpin Shimura datum $(\GSpin(L_\bQ),\cD)$
 admits a hyperspecial level structure at $p$, and admits a Kuga--Satake embedding into a Siegel Shimura $ (\mathrm{GSp}, \mathcal{H}^\pm)$ with hyperspecial level structure at $p$. 
 The corresponding GSpin Shimura variety admits an integral canonical model $\cS$ over $\bZ_p$ with a Kuga--Satake map $\iota:\cS\hookrightarrow \mathscr{A}_g$ into the integral canonical model of a Siegel modular variety. 

Let $\Sigma^{\ddagger}$ be a finite, smooth and complete admissible \textbf{rational polyhedral cone decomposition (rpcd)} for $(\mathrm{GSp}, \mathcal{H}^\pm)$. Then $\mathscr{A}_g$ admits a smooth projective integral toroidal compactification $\mathscr{A}_g^{\Sigma^\ddagger}$ over $\bZ_p$; cf. \cite{MP19}. The restriction of $\Sigma^{\ddagger}$ to $(\GSpin(L_\bQ),\cD)$ is again a finite, smooth complete admissible rpcd, and will be denoted by $\Sigma$. By \cite{MP19}, $\cS$ admits a smooth projective integral toroidal compactification $\cS^{\Sigma}$ over $\bZ_p$, which admits a map $\cS^{\Sigma}\rightarrow \mathscr{A}_g^{\Sigma^\ddagger}$ extending $\iota$, realizing $\cS^{\Sigma}$ as the normalization of the closure of $\im \iota$.  By abuse of notation, this map between compactifications is again denoted by $\iota$. We denote the special fiber $\cS^\Sigma_{\bF_p}$ by $S^\Sigma$. 

\subsubsection{Description of boundary strata} Our reference is \cite{Tay22}. Given $(\GSpin(L_\bQ),\cD)$ as above, a \textbf{cusp label representatives (lcr)} $\Phi$ is a triple $(P_\Phi,\cD^\circ,h)$, where $P_\Phi\subseteq \GSpin(L_\bQ)$ is an admissible parabolic. Let $U_\Phi$ be the unipotent radical of $P_\Phi$, and let $W_\Phi$ be its center. Inside $W_\Phi(\bR)$ there are convex cones $C_\Phi\subseteq C_\Phi^*$, see \cite[\S 2.2]{Tay22} for a description. An admissible rpcd $\Sigma$ as in the last section is an assignment $\Sigma:\Phi\rightarrow \Sigma_\Phi$ where $\Sigma_\Phi$ is a rpcd of the cone $C_\Phi^*$ consists of rational polyhedral cones $\sigma\in W_\Phi(\bR)$ satisfying certain compatibility conditions.

The boundary strata of $\cS^{\Sigma}$ are label by equivalent classes of pairs $(\Phi,\sigma)$, where $\Phi$ is an lcr and $\sigma\in \Sigma_\Phi$ is a rational polyhedral cone whose interior is contained in $C_\Phi$. We denote a boundary stratum by $\cB^{\Phi,\sigma}$. Here is a description of $\cB^{\Phi,\sigma}$: The lcr $\Phi$ determines an integral mixed Shimura variety $\cM_\Phi$, which sits in a tower of mixed Shimura varieties $\cM_\Phi\xrightarrow{p_1} \overline{\cM}_\Phi\xrightarrow{p_2} \cM_\Phi^h$, where $\cM_\Phi^h$ is pure, $p_1$ is a torsor under a torus with cocharacter group $\Gamma_\Phi$ (which is a $\bZ$-lattice in $W_\Phi(\bQ)$), and $p_2$ is a torsor under an abelian scheme. The datum $\sigma$ determines a partial compatification (twist toric embedding) $\cM_\Phi\hookrightarrow \cM_\Phi(\sigma)$ relative over $\overline{\cM}_{\Phi}$, whose boundary component $\cZ_\Phi(\sigma)$ corresponding to $\sigma$, after quotienting a finite group $\Delta_\Phi(\sigma)$, gives $\cB^{\Phi,\sigma}$. In fact, more is true: the formal completion $(\cS^\Sigma)^{/\cB^{\Phi,\sigma}}$, after quotienting the finite group $\Delta_\Phi(\sigma)$, is isomorphic to $\cM_\Phi(\sigma)^{/\cZ_\Phi(\sigma)}$. This is used in \textit{loc.cit} to describe the local structure of a boundary point.  

We can further classify boundary strata into two types, according to the property of $P_\Phi$: \begin{description}
    \item[Type II ]   If $P_\Phi$ is the stabilizer of
a primitive isotropic plane $J_\bQ\subseteq L_\bQ$, then $\Phi$ is said
to be of type II. Following the convention of \textit{loc.cit}, such $\Phi$ is denoted by $\varUpsilon$. There is a unique one dimensional ray $\sigma\in \Sigma_\varUpsilon$, and the corresponding boundary stratum $\cB^{\varUpsilon,\sigma}$ is called a boundary stratum of type II, which is a locally closed divisor in $\cS^\Sigma$. In this case, the tower of mixed Shimura varieties
$\cM_\varUpsilon\rightarrow \overline{\cM}_\varUpsilon\rightarrow \cM_\varUpsilon^h$ admits the following description: $\cM_\varUpsilon^h$ is a modular curve, $p_1$ is of relative dimension 1, and $p_2$ is a torsor under a Kuga--Sato abelian scheme of relative dimension $b-1$. The partial compactification $\cM_\varUpsilon\hookrightarrow\cM_\varUpsilon(\sigma)$ is fiberwise just $\bG_m\hookrightarrow \mathbb{A}^1$. 
\item[Type III] If $P_\Phi$ is the stabilizer of
a primitive isotropic line $I_\bQ\subseteq L_\bQ$, then $\Phi$ is said
to be of type III. Such $\Phi$ is denoted by $\Xi$. 
In the tower of mixed Shimura varieties $\cM_\Xi\rightarrow \overline{\cM}_\Xi\rightarrow \cM_\Xi^h$, $\overline{\cM}_\Xi= \cM_\Xi^h$ is of relative dimension zero over $\bZ_p$. The 
map $\cM_\Xi\rightarrow \overline{\cM}_\Xi$ is a torsor under a torus of rank $b$ cocharacter group $\Gamma_\Xi$. Let $\bar{x}$ be an $\Fpbar$-valued point of the special fiber of $\overline{\cM}_\Xi$. The fiber of the partial compactification $\cM_\Xi\hookrightarrow\cM_\Xi(\sigma)$ over $\bar{x}$ is $\Spec \Fpbar[q_\alpha]_{\alpha\in \Gamma_\Xi^\vee}\hookrightarrow \Spec \Fpbar[q_\alpha]_{\alpha\in \Gamma_\Xi^\vee,\,\la \alpha,\sigma\ra\geq 0}$. While there can be multiple different choices of $\sigma$, of particular importance is the case where  $\sigma$ is an one dimensional inner ray. In this case, the partial compactification $\cM_\Xi\hookrightarrow\cM_\Xi(\sigma)$ is fiberwise $\mathbb{G}_m^{b}\hookrightarrow \mathbb{G}_m^{b-1}\times \bA^1$, and $\cB^{\Xi,\sigma}$ is a locally closed divisor in $\cS^\Sigma$. 
\end{description}
\section{Arithmetic intersection theory}\label{sec:arithmetic}

In this section, we will set up the intersection-theoretic framework required to prove Theorem \ref{mainShimura}. We will also reduce the proof of Theorem \ref{mainShimura} to Theorem \ref{thm: main version 2}. 

\subsection{The global Eisenstein series and arithmetic intersection theory}\label{subsub:gEs}
Let $L$ be as in Section \ref{subsec: set up SV}. Let $q_L(m)$ denote the $m$-th Fourier coefficient of the zero component of the vector-valued Eisenstein series attached to $L$, of weight $1+\frac{b}{2}$ and with constant term $\mathfrak{e}_0$, as in \cite[\S 2.1]{Bru17}. See \cite[Theorem 11]{BK01} for an explicit expression of $q_L(m)$ in terms of local densities. From the expression one can deduce that $|q_L(m)|\asymp m^{\frac{b}{2}}$ for $m$ representable by $(L,Q)$. See also \cite[\S 7.3]{MST22}.

Let $\overline{Z(m)}$ be the Zariski closure of $Z(m)$ in $S^\Sigma$. Let $C$ be a proper smooth connected curve with a non-constant map $C\rightarrow S^\Sigma$ whose image lies generically in $S$. The following result relates the intersection number $C\cdot\overline{Z(m)}$ and the coefficients $q_L(m)$:
\begin{proposition}[\cite{Tay22}, Proposition 4.10]\label{prop: global bound}
 We have $ C\cdot \overline{Z(m)}=|q_L(m)|(C\cdot \omega) + o(m^{\frac{b}{2}})$, where $\omega$ is the tautological line bundle over $S^{\Sigma}$.
\end{proposition}

\subsection{Bounds on local intersections}\label{subsub:maintech}Let $C$ be a proper smooth connected  curve with a non-constant map $C\rightarrow S^\Sigma$ whose image lies generically in the ordinary locus of $S$. 
\begin{defn}
\label{def: intersect at P}
We call a point $P\in C$ a \textbf{boundary point} if it does not map to $S$, and an \textbf{interior point} otherwise. We call an interior point $P\in C$ a \textbf{supersingular point} (resp. \textbf{non-supersingular point}, resp. \textbf{ordinary point}) if it maps to the supersingular locus (resp. {non-supersingular locus}, resp. ordinary locus). The set of supersingular points on $C$ is denoted by $C_{\mathrm{ss}}$. We write $i_{P}(C \cdot \overline{Z(m)})$ for the intersection multiplicity of $C$ and $\overline{Z(m)}$ at $P$. When $C$ is clear from the context, we also write $l_P(m):=i_{P}(C \cdot \overline{Z(m)})$ and call it the local intersection at $P$. 

Let $t$ be a local parameter at $P$, and $\cA_n := \cA_{\Spec k[\![t]\!]/(t^n)}$ is the pullback of $\cA$ to the $n$th infinitesimal neighbourhood of $C$ at $P$. Define $L_{P,n}=L(\mathscr{A}_n)$. It follows directly from the moduli interpretation of $Z(m)$ that \begin{equation}\label{eq:local intersection defn}
    l_P(m)= \sum_{n\geq 1} \#\{v\in L_{P,n}: Q(v) = m \}. 
\end{equation}
\end{defn}


The following result is the main ingredient for Theorem~\ref{mainShimura}:  
\begin{theorem}\label{thm: main version 2}
    Let the setup be as above. Suppose that $C$ does not lie on any special divisor. Let $m$ be coprime to $p$. Then the following are true: \begin{enumerate}
        \item There exists a constant $\alpha < 1$ independent of $m$ such that  $$\sum_{P\in C_{\mathrm{ss}}} l_P(m) < \alpha |q_L(m)| (C\cdot\omega) +o(m^{\frac{b}{2}}).$$
        \item If $P$ is a nonsupersingular interior point, then $l_P(m)=o(m^{\frac{b}{2}})$. 
        \item If $P$ is a boundary point that maps to a stratum labeled by $(\Phi,\sigma)$ with $\sigma$ an one dimensional inner ray, then $l_P(m)=o(m^{\frac{b}{2}})$.
    \end{enumerate} 
\end{theorem}
\begin{proof}[Proof of Theorem \ref{mainShimura} assuming Theorem \ref{thm: main version 2}]  It suffices to show that for any infinite collection $\Delta$ of natural numbers coprime to $p$, the set $\bigcup_{m\in\Delta}(C\cap \overline{Z(m)})(\Fpbar)$ is infinite. We will obtain a contradiction by assuming that there are only finitely many points $P_1\hdots P_n$ of $C$ contained in $\bigcup_{m\in \Delta}\overline{Z(m)}$. By possibly switching to a finer cone decomposition, we can assume that all boundary points on $C$ lies in a stratum described in Theorem \ref{thm: main version 2}(3). Let $m\in \cT$ and let $m\rightarrow \infty$. Summing up the estimates in Theorem \ref{thm: main version 2} and applying Proposition \ref{prop: global bound}, we get 
$$\sum_{P\in \{P_1,...,P_n\}} i_{P}( C\cdot \overline{Z(m)})<\alpha \, C\cdot \overline{Z(m)} +o(m^{\frac{b}{2}}).$$
This yields the desired contradiction.  \end{proof}

The rest of the paper will be devoted to proving Theorem~\ref{thm: main version 2}. In the next subsection, we will reduce Theorem \ref{thm: main version 2}(1) (which is the hardest part) to an estimate of its ``tail''.


\subsection{The head and tail of local intersections}\label{sub: intersection numbers} Let $C\rightarrow S^\Sigma$ be as in the last subsection. For $P\in C$, let $h_P$ be the local intersection multiplicity of $C$ with the non-ordinary locus. Let $g_P(m)=\frac{h_P}{p-1}|q_L(m)|$
be the \textbf{global intersection number of $C$ and $Z(m)$ at $P$}; cf. \cite[Definition  7.6]{MST22}.  Since the non-ordinary locus is the vanishing locus of a section of $\omega^{p-1}$, we have 
\begin{equation}\label{eq:g_Pandglobal}
(C\cdot\omega)|q_L(m)| = \sum_{P\in C}g_P(m).
\end{equation} 

\begin{defn}\label{defn:head and tail}
   Let $\delta\geq 0$. We define the \textbf{$\delta$-head} of $l_P(m)$ to be the sum $$l^{<\delta}_P(m)=\sum_{n< m^\delta}\#\{v\in L_{P,n}:Q(v)=m\}.$$
    We define the
    \textbf{$\delta$-tail} of $i_P(m)$ to be the sum $$l^{\geq \delta}_P(m)=\sum_{n\geq  m^\delta}\#\{v\in L_{P,n}:Q(v)=m\}.$$
\end{defn}

We state two lemmas on the head and tail for a supersingular point. 
\begin{lemma}\label{lm:head} Let $P$ be a  
supersingular point on $C$. There exist a $\delta>0$ and an $\alpha<1$, both independent of $m$, such that $l_P^{<\delta}(m)<\alpha g_P(m)+ o(m^{\frac{b}{2}})$.
\end{lemma}
\begin{lemma}\label{tail} Let $P$ be a  
supersingular point on $C$. Then for any $\delta>0$, $l_P^{\geq \delta}(m)=o(m^{\frac{b}{2}})$.
\end{lemma}

Combining (\ref{eq:g_Pandglobal}), Theorem~\ref{thm: main version 2}(1) follows readily from the two lemmas.

In the following, we prove Lemma~\ref{lm:head} using the main technical results of \cite{MST22}. On the other hand, the proof of Lemma~\ref{lm:head} will be postponed to \S\ref{sec:pvofmain}. 

\begin{proof}[Proof of Lemma~\ref{lm:head}] 
For $n\geq 1$, we enlarge $L_{P,n}$ to a lattice $L_{P,n}\subseteq L'_{P,n} \subset L_{P,n}\otimes \bQ$ maximal away from $p$ and that the decay of $L'_{P,n}$ matches the description given by the decay lemmas (\cite[Theorem 5.2]{MST22} for the supersingular but not superspecial case, and \cite[Theorem 6.2]{MST22} for the superspecial case). For example, in the supersingular but not superspecial case, \cite[Theorem 5.2]{MST22} implies that there is a rank 2 saturated submodule $\Gamma\subseteq L_{P,1}\otimes \bZ_p$ that decays rapidly. Define $L_{P,n}'$ to be maximal away from $p$ and $L_{P,n}'\otimes \bZ_{p}=\Lambda_n\oplus p^{\nu_n}\Gamma$, where $\Lambda_n$ is a direct sum complement of $\Gamma$ and $\nu_n\geq 0$ is biggest integer that $[h_P(1+p+...+p^{\nu_n})p^{-1}]+1\leq n$. One need to carefully choose $\Lambda_n$ so that $L_{P,n}'\supseteq L_{P,n}$, which is always possible. We don't require that $L'_{P,n+1}\subseteq L'_{P,n}$.

Let $\theta_{n}$ denote the theta series attached to $L_{P,n}'$ and write its $q$-expansion as $\theta_{n}(q)=\sum_{m\geq 0}r_n(m)q^m$. Then $\#\{v\in L_{P,n}:Q(v)=m\}\leq r_n(m)$ by definition. Decompose $\theta_n(q)=E_{n}(q)+G_n(q)$, where $E_{n}$ is an Eisenstein series and $G_n$ is a cusp form. Let $q_{n}(m)$ and $g_n(m)$ be the $n$-th Fourier coefficient of $E_{n}$ and $G_n$, respectively. It follows from the computation in \cite[Proposition 7.17]{MST22} that 
$$\sum_{n\geq 1}\frac{q_{n}(m)}{|q_L(m)|} < \alpha\frac{h_P}{p-1}$$
for some constant $\alpha<1$ independent of $m$. In particular, $ \sum_{n\geq 1}q_{n}(m)<\alpha g_P(m)$. To conclude the theorem, it suffices to show that there exists a $\delta>0$ such that $\sum_{n<m^{\delta}}g_n(m)=o(m^{\frac{b}{2}})$. 

For this, one uses a similar argument as in \cite[Proposition 9.1.5]{MAT} to show that there are  constants $N_0,c_0$ only depending on $b$ such that $g_{n}(m)\leq c_0D_{n}^{N_0}m^{0.3b}$, where $D_{n}$ is the discriminant of $L_{P,n}'$ (the key is the uniform bounds on cusp forms: for $b$
 even, one uses Deligne's bound and for $b$ odd, one uses  \cite{waibel18}). On the other hand, decay lemmas imply that $D_n\leq c_1p^{4\log_p n}$ for some constant $c_1$ independent of $n$ (for example, in the supersingular but not superspecial case as above, $D_n=p^{4\nu_n}$). Therefore for any $\delta>0$ and $n<m^\delta$, we have $$g_{n}(m)\leq c_0c_1^{N_0} m^{4N_0\delta+0.3b}.$$
 Suffices to choose $\delta <\frac{b}{5(4N_0+1)}$. 
\end{proof}

\section{Punctual monodromy theorem}
In this section, we prove a local-global theorem in the context of special endomorphisms. Specifically, we prove the following result. 

\begin{theorem}\label{theorem: algebraic}
Let $(X,x)$ be a smooth pointed variety over $\Fpbar$ and let $f:(X,x)\rightarrow S^\Sigma$ be a morphism whose image lies generically in the ordinary stratum. If the pullback $p$-divisible group over $X$ admits a formal special endomorphism over $X^{/x}$, then the image of $X$ is contained in a Shimura subvariety. Further, let $\Lambda$ denote the module of formal special endomorphisms on this pullback. Then, the smallest Shimura subvariety containing $X$ has codimension equal to rank $\Lambda$. 
\end{theorem}

We recall the definition of when a (possibly transcendental) field extension is separable. 
\begin{definition}\label{def:separable}
    We say that a (possibly transcendental) field extension $L/K$ is separable if every $x\in L$ that is algebraic over $K$ is separable over $K$.
\end{definition}

The main technical result that we will need is the following.

\begin{lemma}\label{lem: algebraicity over fields}

    Let $K$ be a field, and let $\pdiv/K$ be an ordinary one-dimensional $p$-divisible group. Let $\pdiv^{\mult}, \pdiv^{\et}$ be its connected sub and etale quotient respectively, and suppose that its $p$-adic monodromy is semisimple (equivalently, $\pdiv^{\et}$ is a semisimple etale $p$-divisible group over $K$). Suppose there exists a separable field extension $L/K$ and a saturated etale sub $p$-divisible group $\pdiv_1^\et \subset \pdiv_L^\et$ such that $\pdiv_1^\et \times \pdiv_L^{\mult} \subset \pdiv_L$. Then, there exists a saturated etale sub $p$-divisible group $\pdiv_2^\et \subset \pdiv^{\et}$ defined over $K$ that splits off as a factor of $\pdiv$ and such that ${\pdiv}_{2,L}^\et \supset \pdiv^\et_{1}$. 
\end{lemma}

\subsection{Canonical pairing and the proof of Lemma \ref{lem: algebraicity over fields}}
Before moving on to the proof of Lemma \ref{lem: algebraicity over fields}, we recall work of Chai \cite[Section 2]{Ch03}. Instead of working in full generality, we will recall Chai's work in our setting of one-dimensional ordinary $p$-divisible groups over a field. By twisting $\pdiv/K$ by an etale character, we may assume that $\pdiv^{\mult}\simeq \mu_{p^{\infty}}$ without losing generality. 

Chai in \cite[Section 2]{Ch03} defines an etale sheaf $\nu_{p^{\infty}} := \varprojlim_n \coker([p^n]:\bG_{m}\rightarrow \bG_{m})$. At the level of points, $\nu_{p^{\infty}}(K) = \varprojlim K^*/(K^*)^{p^n}$. The data of $\pdiv$, an ordinary one-dimension $p$-divisible group with multiplicative part isomorphic to $\mu_{p^{\infty}} $ and etale part isomorphic to $\pdiv^\et$, is equivalent to a $\Gal(K^{\sep}/K)$-equivariant morphism of $\bZ_p$-modules 
\begin{equation}\label{eq:qpair}
q(\pdiv): T_p(\pdiv_{K^{\sep}}^{\et}) \rightarrow \varprojlim K^{\sep, *}/(K^{\sep,*})^{p^n}.
\end{equation}

 Equation \eqref{eq:qpair} is functorial in the field $K$. This equivalence follows from Kummer theory and the fact that ordinary $p$-divisible groups are classified by extensions of etale $p$-divisible groups by connected $p$-divisible groups. For more details, we refer the reader to \cite[Section 2]{Ch03}. 

  We are now ready to prove Lemma \ref{lem: algebraicity over fields}. 
 \begin{proof}
        We first note that any $p$-divisible subgroup of $\pdiv^{\prime,\et} \subset \pdiv^{\et}$ defined over $K$ gives rise to a one-dimensional connected ordinary $p$-divisible subgroup $\pdiv' \subset \pdiv$. This is simply by pulling back the connected etale exact sequence 
        \[
0 \longrightarrow \pdiv^{\mult}
\longrightarrow \pdiv
\longrightarrow \pdiv^{\et}
\longrightarrow 0
\]
        via the map $\pdiv^{\prime,\et}\subset \pdiv^{\et}$. 

    Secondly, given any one-dimensional ordinary $p$-divisible group $\pdiv'$ with connected part $\mu_{p^{\infty}}$, we note that $q(\pdiv^{\prime,\et}) = 1$ if and only if $\pdiv' \simeq \pdiv^{\prime,\et} \times \mu_{p^{\infty}}$.

     Now, consider the kernel $T_2 := \ker q(\pdiv)$. This gives rise to a saturated $p$-divisible subgroup $\pdiv_2^{\et} \subset \pdiv^{\et}$ defined over $K$, as $q(\pdiv)$ is Galois-equivariant. We therefore have that $\pdiv_2^{\et}$ splits off as a direct factor of $\pdiv$ by the previous two observations, as well as the fact that the $p$-adic monodromy of $\pdiv^{\et}$ is semisimple. 
     
    Consider the map $q(\pdiv_{L})$. We claim that $\ker q(\pdiv_L) = \ker q(\pdiv)$. Indeed, we have that the extension $L/K$ and therefore $L^{\sep}/K^{\sep}$ are separable. Therefore, the map  $K^{\sep,*}/(K^{\sep,*})^{p^n}\to L^{\sep,*}/(L^{\sep,*})^{p^n}$ is injective for every $n$, whence the claim. We have that $T_p(\pdiv_1^{\et})\subset \ker q(\pdiv_L)$. It follows that $\pdiv^{\et}_{2,L} \supset \pdiv^{\et}_1$. The lemma now follows. 
    
     \end{proof}

\subsection{Proof of punctual monodromy.}
We now show that Lemma \ref{lem: algebraicity over fields} implies Theorem \ref{theorem: algebraic}:
\begin{proof}
We pick $K$ to be the function field of $X$, $L$ to be the fraction field of the {complete} local ring of $X$ at $x$ and $\pdiv/K$ to be the extended Brauer group. The existence of $\Lambda$  at $X^{/x}$ implies that the maximal etale subgroup $\pdiv^{\et}_1 \subset \pdiv_L^{\et}$ that splits off as a direct factor of $\pdiv_L$ has height equal to rank $\Lambda$. We have that $L/K$ is separable, and also that the $p$-adic monodromy of $\pdiv^{\et}/K$ is semisimple by \cite{MD20}, and we may hence apply Lemma \ref{lem: algebraicity over fields} to deduce that $\pdiv_1^{\et} \subset \pdiv^{\et}$ is defined over $K$ and that $\pdiv/K \simeq \pdiv_1^{\et} \times \pdiv'$. We immediately deduce that the image $X^{/x}$ is contained in a formal subtorus of $(S^\Sigma)^{/x}$ having codimension equal to rank $\Lambda$. The result now follows from \cite{Ruofan}.

\end{proof}





    
    


\section{The Basic Mumford-Tate conjecture and $\overline{\bQ}$-algebraicity of formal special endomorphisms}\label{sec: Mumford Tate}

We work in the following setup. Let $C \subset S$ be a smooth curve contained in the ordinary locus. Let $\cS'$ be the smallest sub-Shimura variety of $\cS$ containing $\cS$, and let $F$ be the number field associated to $\cS$ as in Section \ref{subsec: subshimura varieties} - recall that  $F$ is either a totally field of a CM extension of a totally field. The action of $F$ on $\bL_{\cris}/C$ induces a decomposition $\bL_{\cris}\otimes \bQ_p = \bigoplus_{v|p} \bL_{\cris,v}$ where $v$ ranges of the places of $F$ dividing $p$. Let $v_0$ denote the place corresponding to the embedding $S'\subset S$. We note that the isocrystal $\bL_{\cris}' :=\bigoplus_{v\neq v_0} \bL_{\cris,v}$ is a unit-root isocrystal.

\begin{theorem}\label{thm: Qbar algebraicity of formal special endomorphisms}
    Let $x\in C(\bF_q)$ be a point, let $L_x\subset \End(A_x)$ be the module of special endomorphisms at $x$. Then, the $L_x\otimes W(\bF_q) \cap \bL_x'$ descends to a finite extension of $\bQ$ with bounded degree. Consequently, if $\Lambda_x$ is the set of formal special endomorphisms at the formal neighbourhood of $C$ at $x$, then the subspace $\Lambda_x[1/p] \subset L_x \otimes \bQ_p$ descends to a finite extension of $\bQ$ with bounded degree. 
\end{theorem}
We will need to distinguish the supersingular case with the non-supersingular case.
\subsection{Non-supersingular points}

We start with the following lemmas. 
\begin{lemma}\label{lem:nonsslifting}
    Let $x$ be a non-supersingular point as above. Then, there is a lift $\tilde{x}$ to a finite extension of $W(\bF_q)$ such that 
    \begin{enumerate}
        \item All the special endomorphisms $L_x$ lift. 
        \item The real-multiplication / unitary structure also lifts to $\tilde{x}$. 
    \end{enumerate}
\end{lemma}
\begin{proof}
    The first part follows from the observation that all the special endomorphisms in $L_x$ lift to any characterstic zero lift such that the slope-filtration of the filtered isocrystal admits a splitting in the category of filtered isocrystals. 

    For the second part, in order to preserve the extra structure that $\cO_F$ introduces, it suffices to preserve it upto tensoring $F$ by $\bZ_p$. Then, it just becomes a result about lifting $F\otimes \bZ_p$ action for a Lubin-Tate group. But this can be done by classical Lubin-Tate theory. 
\end{proof}

Similar to the decomposition of $\bL_{\cris}$, the Betti local system restricted to $\ShimKprime$ satisfies $\bL_{B} \otimes \bC = \bigoplus_{w} \bL_{B,w}$, where $w$ ranges over the archimedian places of $F$. The fact that $F$ is also the reflex field gives us an embedding $F \rightarrow \bC$, and therefore picks out an archimidean place $w_0$ of $F$.  Define $\bL'_{B} := \bigoplus_{w\neq w_0} \bL_{B,w}$. Note that each subspace $\bL_{B,w}$ descends to the normal closure of $F$.
\begin{lemma}\label{lem: complex algebraicity}
    Let $\ShimKprime \subset \ShimK$ be as above. Let $y\in \ShimKprime(\bC)$ be a point, and let $L_y$ be the $\bZ$-module of special endomorphisms at $y$. Then, the subspace $\bL'_{B}\cap (L_y\otimes \bC)$ descends to the normal closure $\tilde{F}$ of $F$.
\end{lemma}
\begin{proof}
This follows directly from the fact that each of the $\bL_{B,w}$ descends to a bounded-degree extension of $\bQ$. 
\end{proof}

We are now ready to prove Theorem \ref{thm: Qbar algebraicity of formal special endomorphisms} in this case.
\begin{proof}[Non-supersingular case.]

Lift $x$ to $y$ defined over $K$, which we may do by Lemma \ref{lem:nonsslifting}. Now pick an embedding $K$ to $\bC$. Do this so that the associated complex embedding of $F$ is induced by $w_0$. 

The crystalline-deRham comparison gives you an isomorphism $\bL_{\dR,y} \simeq \bL_{\cris,x}\otimes K$. This isomorphism respects the $F$ vector space structure, and preserves special endomorphisms. So, it suffices to prove the theorem for $y$ but with $L_x = L_y$ and $\bL_{\dR}$ in place of $\bL_{\cris}$. Then Betti-deRham comparison over $\bC$ reduces algebraicity to the Betti question. But this is just Lemma \ref{lem: complex algebraicity}.
%


\end{proof}

\subsection{The supersingular case and the proof of Theorem \ref{thm: MT GSpin}}

\subsubsection{Basic points}
We first formulate our characteristic $p$ analogue of the Mumford-Tate conjecture where the ambient Shimura variety need not contain a supersingular point. Instead, we work with the basic locus. Our setting is now $x\in X\subset S$, where $S$ is a mod $p$ Shimura variety with good reduction and $x$ is a point in the Basic locus. We suppose that $X$ is geometrically irreducible. Let $G/\bQ$ denote the reductive group defining $S$. The Lie Algebra $\fG$ of $G$, when treated as a $G$-representation gives rise to $\ell$-adic local systems $\fG_{\et,\ell}$ for $\ell\neq p$ and also to an $F$-isocrystal $\fG_{\cris}$. We will use these objects in place of $\End(V_{\et,\ell}), \End(V_{\cris})$ to formulate our conjecture. Let $H_\ell \acts \fG_{\et,\ell,x}$ be defined identically as in the supersingular case, and define $H_p$ analogously. 

Let $I_x$ be as in \cite{MKmodp}. We have canonical embeddings $I_{x,\bQ_\ell} \subset G_{\bQ_\ell}$ and $I_{x,\Qpun} \subset G_{\Qpun}$. As $x$ is basic, \cite{MKmodp} shows that these embeddings are isomorphisms. Therefore, $\Lie I_x$ gives $\fG_{\et,\ell,x}$ and $\fG_{\cris,x}$ a $\bQ$-structure. Let $H$ denote the smallest $\bQ$-group (with respect to this structure) that satisfies $H_{\bQ_\ell} \supset H_\ell$ and $H_{\Qpun}\supset H_p$. 

\begin{conj}\label{conj: Basic MT}
    We have $H_{\bQ_\ell} = H_\ell$, and $H_{\Qpun} = H_p$.
\end{conj}

\begin{remark}
\begin{enumerate}
    \item In upcoming work of Madapusi-Lee, the authors define a $\bQ$-group of quasi-isogenies $I_x$ that generalizes Kisin's definition to the setting of all Shimura varieties mod $p$ with good reduction. Combined with the semisimplicity results of \cite{BST}, one deduces that $I_x$ is a $\bQ$-form of $G_{\bQ_\ell}, G_p$ as in the Hodge type case, and one may therefore formulate Conjecture \ref{conj: Basic MT} to the setting of all Shimura varieties with good reduction mod $p$.
    \item We note that Conjecture \ref{conj: Basic MT} immediately implies Conjecture \ref{conj: ss MT}.
\end{enumerate}
      
\end{remark}

\subsubsection{}

We will first prove our group-theoretic Mumford-Tate conjecture and use that to prove Theorem \ref{thm: Qbar algebraicity of formal special endomorphisms} in the supersingular case. We first start by observing that Conjecture \ref{conj: Basic MT} (and therefore Conjecture \ref{conj: ss MT}) is true when $X$ is a connected component of a Shimura variety with good reduction and $x\in X$ is a basic point. As we will see, this follows directly from \cite{MKmodp}.
\begin{proposition}[Kisin]\label{prop: MT Shimura}
Conjecture \ref{conj: Basic MT} is true when $X$ is a SV with good reduction.
\end{proposition}
\begin{proof}
    This follows immediately using the results of \cite{MKmodp}. Let $G$ dnote the reductive group defining our Shimura variety. The group $H_\ell$ is simply isomorphic to $G^{\ad}_{\bQ_{\ell}}$, and $H_p$ is isomorphic to $G_{\Qpun}^{\ad}$. The group $I^{\ad}_x$ clearly satisfies $I_{x,\bQ_\ell}^{\ad} \subset H_{\ell} $ by definition (with the analogous $p$-adic statement also holding). As $x$ is basic, the reverse inclusion also holds, whence the theorem.
    
\end{proof}

We are now ready to prove Theorem \ref{thm: MT GSpin}. 
\begin{proof}
    We will work in the crystalline setting, as the $\ell$-adic argument is identical. 
    Without loss of generality, we will assume that $X\subset S \subset \Ag$ where $S$ is a $\GSpin$-Shimura variety with good reduction\footnote{We may do this at the cost of embedding a quadratic lattice into a larger-rank one.} associated to a quadratic lattice $L$ self-dual at $p$ and having signature $(b,2)$. The main theorem of \cite{Ruofan} together with Proposition \ref{prop: MT Shimura} gives us our result immediately if $S$ is the smallest Shimura variety containing $X$. 

    Suppose not. Let $S'$ denote the smallest Shimura variety containing $X$. By \cite{Ruofan}, the monodromy of $A/X$ equals that of $A/S'$. The sub-Shimura variety $S'\longrightarrow S$ is defined by endomorphisms. Specifically, there is a $\bQ$-algebra $E$ along with an action $E\acts \Cl(L)$ such that the reductive group defining $S'$ is the centralizer of $E$ in $\GSpin(L)$. This has the consequence that there is an embedding $E\subset \End(A_x)$ compatible with the embedding $E_{\cris}\subset \End(H_{\cris}^1(A_x))$.
    
    Let $H_p\subset \GSpin\acts \End(H^1_{\cris}(A_x))$ be the crystalline monodromy group associated to the family $A/X$. By \cite{Ruofan}, $H_p$ is the centralizer in $\GSpin$ of $E_{\cris}$. Consider the monodromy action of $H_p\subset \GSpin$ on the isocrystal $\End(H^1_{\cris}(A_x))$. We have a commutative diagram
\[
\begin{tikzcd}[column sep=tiny]
\GSpin & \acts & \End(H^1_{\cris}(A_x)) & \supset & E_{\cris} \\
I_x \arrow[u, hook] & \acts & \End(A_x) \arrow[u, hook] & \supset & E \arrow[u, hook]
\end{tikzcd}
\]
It follows that the subgroup $I_{H,x}\subset I_x$ that commutes with $E$ is a $\bQ$-form of $H_p$, compatible with the action above. This proves Theorem \ref{thm: MT GSpin} as required. 
\end{proof}

We are now ready to prove that the module of formal special endomorphisms descend to a $\overline{\bQ}$-vector space. 

\begin{proof}[Proof of Theorem \ref{thm: Qbar algebraicity of formal special endomorphisms} for supersingular points.]
This will be a relatively straightforward corollary to Theorem \ref{thm: MT GSpin}. We have a supersingular point $x\in S(\Fpbar)$. 

With notation as in the proof of Theorem \ref{thm: MT GSpin}, consider the following diagram:
\[
\begin{tikzcd}[column sep=tiny]
H_p & \acts & \End(H^1_{\cris}(A_x)) & \supset & \bL_{\cris} \\
I_{H,x} \arrow[u, hook] & \acts & \End(A_x) \arrow[u, hook] & \supset & L_x \arrow[u, hook]
\end{tikzcd}
\]

The vector space $\bL'_x[1/p]$ is simply a sub-representation of $H_p$. We also have that $\bL_x[1/p]$, as a representation of $H_p$, has the property that every irreducible that shows up has multiplicity one. The above diagram and Theorem \ref{thm: MT GSpin} implies that it must therefore be induced by a subrepresentation of $I_{H,x}$, and therefore descends to a subspace defined over $\overline{\bQ}$. The result follows.

\end{proof}

\section{Local structure of GSpin Shimura varieties} In Sections 6--8, let $k$ be an algebraically closed field of characteristic $p$, and set $W=W(k)$ and $K=W[1/p]$. In Theorem~\ref{thm:globalanalysis} and its proof, we assume that $k=\Fpbar$. We set up notation for local computation over GSpin Shimura varieties. As already explained in \cite{MST22}, we will assume that $b=2m$ is even and $L\otimes\bZ_p$ is not split.

Let $P\in S(k)$. Let us write the K3 crystal $\bL_{\cris,P}$ simply as $\bL_P$. Recall that $\bL_P \otimes_W k = \bL_{\dR, P}$ is equipped with a Hodge filtration. Let $\mu : \mathbb{G}_{m,W} \to \text{SO}(\bL_P)$ be a co-character whose reduction modulo $p$ splits the Hodge filtration. Define $U^{\mathrm{op}}$ to be the opposite unipotent subgroup in $\mathrm{SO}(\bL_P)$ relative to $\mu$, and let $\Spf(R)$ denote the completion of $U^{\mathrm{op}}$ at the identity section. Let $\sigma$ be a lift of the Frobenius endomorphism on $R \otimes_W k$, and let $u \in U^{\mathrm{op}}(R)$ be the tautological point. According to \cite[\S 1.4, 1.5]{KM09} (see also \cite[Proposition 4.7]{KLSS}), there exists an isomorphism $\cS^{/P} \simeq \Spf(R)$ through which $\bL_\cris(R)$ can be identified with $\bL_P \otimes_W R$ such that the Frobenius action is given by $\Frob = u \circ (\varphi \otimes \sigma)$.

Define $\cL:=\bL^{\varphi=1}_P$, the Frobenius invariant $\mathbb{Z}_p$-lattice in $\mathbb{L}_P$. Note that $\mathrm{rk}\cL=b+2$ when $P$ is supersingular, and equals $b+2-2h$ when $P$ is non-supersingular of height $h$. 

\subsection{Supersingular case}\label{superspecial 3} 
Suppose that $P$ is superspecial. Following \cite{MST22} (see also \cite[\S 4.2.1]{JSY26}), we can pick a basis $\cL=\Span_{\bZ_p}\{e', f', e_i, f_i \}_{i = 1, \cdots, m}$, pick coordinates $$R=W[\![x_1,...,x_m,y_1,...,y_m]\!]$$
with $\sigma(x_i)=x_i^p,\sigma(y_i)=y_i^p$, 
as well as an element $\lambda \in W(\bF_{p^2})^{\times}$ satisfying $\sigma(\lambda) = -\lambda$, 
such that \begin{enumerate}
   \item the intersection matrix of the chosen basis of $\cL$ is $$\left[
\begin{array}{cc|cc}
2p & & \\
&-2\lambda^2p\\
\hline
&&0&I_m\\
&&I_m&0
\end{array}\right].$$
\item a basis of $\bL_P$ is given by $e_i,f_i$, $1\leq i\leq m$, together with \begin{equation}\label{eq:changeof basis}
    w' = \frac{1}{2\lambda p}(\lambda e' + f'), v' = \frac{1}{2\lambda}(\lambda e' - f').
\end{equation} 
    \item  $\Frob=(1+F)\circ \sigma$, where
\begin{equation*}
\renewcommand{\arraystretch}{1.5}
    F = \left[
    \begin{array}{cc|cc}
         \frac{Q}{2p} & \frac{-\lambda Q}{2p} & \frac{\bfx}{2p} & \frac{\bfy}{2p} \\
         \frac{Q}{2p\lambda} & \frac{-Q}{2p} & \frac{\bfx}{2p \lambda} & \frac{\bfy}{2p \lambda} \\ \hline
         -\bfy^t & \lambda \bfy^t & & \\
         -\bfx^t & \lambda \bfx^t & & 
    \end{array}\right],
\end{equation*}
in which $\mathbf{x}=[x_1,...,x_m]$, $\mathbf{y}=[y_1,...,y_m]$ and $Q=-\sum_{i=1}^m x_iy_i$. 
\end{enumerate}
One can also write down similar explicit coordinates for general supersingular $P$, see \cite{MST22}.

\subsection{Non-supersingular case}\label{sub:nons_setup}
Let $P$ be non-supersingular with height $h\geq 1$. We can pick a basis $\bL_P=\Span_{W}\{e_0,f_0, e_1,f_1,...,e_{m},f_m\}$ such that 
\begin{enumerate}
    \item $\la e_i,f_j\ra=\delta_{ij}$,
    \item $\mu: t\rightarrow \mathrm{diag}[t^{-1},t,1,...,1]$,
    \item $\varphi$ acts as 
$$e_{h-1}\xrightarrow{\varphi}  e_{h-2}\xrightarrow{\varphi} \cdots\xrightarrow{\varphi}  e_1\xrightarrow{\varphi}   e_0\xrightarrow{\varphi} p^{-1}e_{h-1},$$  $$f_{h-1}\xrightarrow{\varphi}  f_{h-2}\xrightarrow{\varphi} \cdots\xrightarrow{\varphi} f_1\xrightarrow{\varphi}   f_0\xrightarrow{\varphi} p^{-1}f_{h-1},$$ and acts as identity on $e_i,f_i$ for $i\geq h$. 
\end{enumerate}  

Under above conventions, we have 
\begin{equation*}
    \varphi=
     \left[
    \begin{array}{cc|ccc|cccc}
          & & &&&&&&\\
          & &&I_{2(h-1)}&&&&& \\ 
          & &&&&&&&\\ 
          \hline
        p^{-1} & &&&&& \\
        &p&&&&\\\hline
        &&&& &&&&\\ 
         &&&& &&I_{2(m-h+1)}&\\ &&&&&&&&\end{array}\right]\sigma. 
\end{equation*}
We again choose coordinates 
\begin{equation}
    \label{eqn: coords of R} 
     R \simeq W[\![x_1, \cdots, x_m, y_1, \cdots, y_m]\!]
\end{equation}
with $\sigma$ sending $x_i$ to $x_i^p$ and $y_i$ to $y_i^p$, such that 
\begin{equation*}
   u = \left[
    \begin{array}{cc|ccccc}
         1&Q &x_1&y_1&\cdots &x_m&y_m\\
        & 1 \\
        \hline
       & -y_1\\
       &-x_1\\
       &\vdots &&&I_{2m}\\ 
       &-y_{m}\\
       &-x_m\\
    \end{array}\right],\,\,Q=\sum_{k=1}^mx_ky_k.
\end{equation*}
Then 
\begin{equation}\label{eq:Frob_parts}
    \Frob=u\circ(\varphi\otimes\sigma)= \begin{bmatrix}
        A&B\\
        D&I_{2(m-h+1)}
    \end{bmatrix}\sigma, 
\end{equation}
where $A$ is the top left $2h\times 2h$ square matrix such that:
\begin{equation*}
    \begin{aligned}
       & A=\left[
    \begin{array}{cc|cc|ccccc}
       p^{-1}x_{h-1}&py_{h-1}&  1&Q &x_{1}&y_{1}&\cdots &x_{h-2}&y_{h-2}\\
        & &&1 \\
         \hline
     &&  & -y_1\\
      && &-x_1\\
       &&&\vdots &&&I_{2(h-2)}\\
      && &-y_{h-2}\\
       && &-x_{h-2}\\\hline
 p^{-1}  &&&-y_{h-1}\\
       &p&&-x_{h-1}\\
    \end{array}\right],\;\;\;\text{when }h\geq 2\\
    &A=\begin{bmatrix}
    p^{-1} &pQ\\
    &p
\end{bmatrix},\;\;\;\text{when }h=1,
    \end{aligned}
\end{equation*}
$B$ is a $2h\times 2(m-h+1)$ matrix with top row $[x_{h},y_{h}
,\cdots
,x_{m},y_m]$ and 0 elsewhere; $D$ is a $2(m-h+1)\times 2h$ matrix with fourth (resp. second) column $[-y_{h},-x_{h}
,\cdots
,-y_{m} ,-x_m]^t$ and 0 elsewhere if $h\geq 2$ (resp. $h=1$).


\section{Eventually rapid decay in the interior}
Let $P\in S(k)$. Consider a formal curve $\cC:\Spf k[\![t]\!]\rightarrow S^{/P}$ whose (rigid) generic point maps to the ordinary stratum. Let $\cL=\bL^{\varphi=1}_P$.
\begin{defn}\label{def:eventualdecay}
 Let $v\in \cL$ be a vector, and let $\Gamma\subseteq \cL$ be a  saturated sublattice. \begin{enumerate}
     \item We say $v$  \textbf{eventually decays}, if no $p$-power multiple of $v$ deforms to $\cC$,
     \item  We say $v$ \textbf{eventually rapidly decays}, if there exists a positive integer $D$ such that for all $n\geq  0$, $p^{n}v$ at most lifts to $\mathrm{Spf}\,k[\![t]\!]/(t^{(1+p+p^2+...+p^{n})D})$. 
     \item We say $\Gamma$ eventually decays, if every primitive vector of $\Gamma$ eventually decays. 
     \item We say $\Gamma$ {eventually rapidly decays}, if there is a uniform positive integer $D$ such that for every primitive vector $v\in \Gamma$ and every $n\geq 0$, $p^{n}v$ at most lifts to $\mathrm{Spf}\,k[\![t]\!]/(t^{(1+p+p^2+...+p^{n})D})$. 
 \end{enumerate}
\end{defn}
The main goal of this section is the following:
\begin{theorem}\label{thm:eventually_rapid_decay} If $\Gamma\subseteq \cL$ eventually decays, then it eventually rapidly decays.
 \end{theorem}
 This will be proved in the rest of this section. 

 \begin{corollary}\label{cor:control_subspace}
For $N\geq 1$, let $\cL_N\subseteq \cL$ be the sublattice of vectors that lift to $\mathrm{Spf}\,k[\![t]\!]/(t^N)$, and let $\Lambda\subseteq \cL$ be the saturated sublattice that consists of vectors that do not eventually decay. Then there is a positive integer $D$ such that for any $n\geq 0$ we have $\cL_{(1+p+p^2+...+p^{n})D+1}\subseteq \Lambda+p^{n+1}\cL$.
 \end{corollary}
\begin{proof}
Write $\cL=\Lambda\oplus \Gamma$. Then $\Gamma$ eventually decays. Apply Theorem~\ref{thm:eventually_rapid_decay} to conclude. 
\end{proof}

 \subsection{Deforming special endomorphisms} We first fix some notation:
\begin{enumerate}
    \item
Given a formal special endomorphism $v \in \cL$, there exists a single series $x_v\in R\otimes k$ whose vanishing is the locus to which $v$ extends; cf. \cite[Corollary 5.17]{MP16}. Such series is well defined up to a unit. 
\item By \cite[\S 1.4, 1.5]{KM09}, or simply by Dwork's trick, a $v\in \cL$ uniquely extends to a Frobenius invariant horizontal section $\tilde{v}\in\bL_P\otimes (R\widehat{\otimes} K)$. In fact, $\tilde{v}$ can be explicitly obtained from $v$ by iterating the Frobenius: 
$$\tilde{v}=\lim_{N\rightarrow \infty} \Frob^N v.$$
\end{enumerate}

The following liftability criterion will be useful throughout:
\begin{lemma}\label{T:PDlemma} Let $\iota:\Spf k[\![t]\!]\rightarrow \Spf R\otimes k$ be a map. Lift $\iota$ to a map $\tilde{\iota}:\Spf W[\![t]\!]\rightarrow \Spf R$. Let $0\leq s\leq \infty$ be an integer. An element $v\in \cL$ deforms to $\mathrm{Spf}\,k[\![t]\!]/(t^s)$ if and only if $\tilde{\iota}^*\tilde{v}\in\bL_P\otimes W[\![t,\frac{t^{ps}}{p}]\!]$.
\end{lemma}
\begin{proof}
    The proof is identical to \cite[Lemma 5.1.7]{JSY26}. For reader's convenience, we sketch the proof here. First, note that $ W[\![t,\frac{t^{ps}}{p}]\!]$ is the $p$-adic completion of the PD envelope of $R$ with respect to $(p,t^s)$. By \cite[\S 2.2, 2.3]{DJ95} and \cite[Theorem 4.6]{DJ99}, there is a natural fully faithful functor from the category of $p$-divisible groups over $k[\![t]\!]/(t^s)$ to Dieudonné modules over $ W[\![t,\frac{t^{ps}}{p}]\!]$. As a result, $v$ deforms to $\Spf(k[\![t]\!]/t^s)$ if and only if it extends to a horizontal section of $\bL_P\otimes W[\![t,\frac{t^{ps}}{p}]\!]$, if and only if $\tilde{\iota}^*\tilde{v}\in \bL_P\otimes W[\![t,\frac{t^{ps}}{p}]\!]$.\end{proof}

\subsection{Supersingular case} We will only work out the most difficult case, i.e., when $P$ is superspecial. The general supersingular case follows from a similar strategy and is omitted. We will use the explicit coordinates in  \S\ref{superspecial 3}. To compute $\tilde{v}$, the key point is to compute $$F_\infty := \lim_{N\rightarrow \infty}\Frob^N=\prod_{i\geq 0}(1+F^{(i)}).$$
In \cite[\S 5]{JSY26}, we explicitly computed $F_\infty$ with the further assumption that $Q=0$. Now we drop this assumption.
The treatment is parallel to that of \textit{loc.cit}.   
\subsubsection{Computing $F_\infty$}
For $i\geq 1$, let 
$$Q_{i}=-\sum_{k=1}^m\left( x_ky^{(i)}_k+x^{(i)}_ky_k \right).$$
Also let $Q_0=Q$. Then the upper left, upper right, and lower left blocks of $F$ can be written as \begin{equation}\label{eq:blocks}
 \frac{Q_0}{2p}\begin{bmatrix}
    1\\
   \lambda^{-1}
\end{bmatrix}\begin{bmatrix}
    1 & -\lambda
\end{bmatrix}  ,\,\,\, \frac{1}{2p}\begin{bmatrix}
    1\\
   \lambda^{-1}
\end{bmatrix}\begin{bmatrix}
    \mathbf{x} & \mathbf{y}
\end{bmatrix},\text{ and   }\begin{bmatrix}
    -\mathbf{y}^t\\
    -\mathbf{x}^t
\end{bmatrix}\begin{bmatrix}
    1 & -\lambda
\end{bmatrix},
\end{equation} respectively. The following lemma is elementary: 
\begin{lemma}\label{lm:elementarycomp}
  For all $i> 0$ and $j\geq 0$, we have \begin{enumerate}
      \item $\begin{bmatrix}
    1 & -\lambda
\end{bmatrix}^{(j)}\begin{bmatrix}
    1\\
   \lambda^{-1}
\end{bmatrix}^{(i+j)}=1-(-1)^{i},$.
\item $
    \begin{bmatrix}
    \mathbf{x} & \mathbf{y}
\end{bmatrix}^{(j)}\begin{bmatrix}
    -\mathbf{y}^t\\
    -\mathbf{x}^t
\end{bmatrix}^{(i+j)} = Q_{i}^{(j)}$.
  \end{enumerate}
  
\end{lemma}

\begin{notation}\label{not:admissible}
An admissible index system  of length $n$ is the following matrix $\binom{j_1 \,j_2\,...\,j_n}{i_1\,i_2\,...\,i_n}$
where all entries are non-negative integers, and $j_{k+1}>i_{k}+j_k$ for $k=1,...,n-1$. We will define  $$Q_{\binom{j_1 \,j_2\,...\,j_n}{i_1\,i_2\,...\,i_n}}= \prod_{k=1}^n Q_{i_k}^{(j_k)},\;\;\;\;\;\delta_{\binom{j_1 \,j_2\,...\,j_n}{i_1\,i_2\,...\,i_n}}= \prod_{k=1}^{n-1} \frac{1}{2}[1-(-1)^{j_{k+1}-i_k-j_k}].$$
When $n=0$, we also define by default that $Q_{\varnothing}=1, \delta_{\varnothing}=1$. 
\end{notation}
\begin{lemma}\label{lm:Finftyexplicit} 
Notation as above. We have
 \begin{align*}
    F_\infty= I+\sum_{n\geq 1}p^{-n}F_n,
\end{align*}
where $F_n=\begin{bmatrix}
X_n & Y_n\\
Z_n & W_n
\end{bmatrix}$, and $X_n$ is the upper left $2\times 2$ block, $Y_n$ is the upper right $2\times 2m$ block, $Z_n$ is the lower left $2m\times 2$ block, and $W_n$ is the lower left $2m\times 2m$ block, with explicit description as follows (all matrices in the summation are admissible in the sense of Notation~\ref{not:admissible}):
\begin{align*}
X_n&=\sum_{\binom{j_1 \,j_2\,...\,j_n}{i_1\,i_2\,...\,i_n}} \frac{1}{2} \delta_{\binom{j_1 \,j_2\,...\,j_n}{i_1\,i_2\,...\,i_n}}Q_{\binom{j_1 \,j_2\,...\,j_n}{i_1\,i_2\,...\,i_n}}\begin{bmatrix}
    1\\
   (-1)^{j_1}\lambda^{-1}
\end{bmatrix}\begin{bmatrix}
    1 & (-1)^{i_n+j_{n}+1}\lambda
\end{bmatrix},\\
   Y_n&= \sum_{\binom{j_1 \,j_2\,...\,j_{n-1}\,j_n}{i_1\,i_2\,...\,i_{n-1}\,0}} \frac{1}{2} \delta_{\binom{j_1 \,j_2\,...\,j_{n-1}\,j_n}{i_1\,i_2\,...\,i_{n-1}\,0}}Q_{\binom{j_1 \,j_2\,...\,j_{n-1}}{i_1\,i_2\,...\,i_{n-1}}}\begin{bmatrix}
    1\\
   (-1)^{j_1}\lambda^{-1}\end{bmatrix}\begin{bmatrix}
   \mathbf{x} & \mathbf{y}
\end{bmatrix}^{(j_n)},\\
  Z_n&= \sum_{\binom{j_0\,j_1 \,j_2\,...\,j_n}{0\,i_1\,i_2\,...\,i_n}} \delta_{\binom{j_0\,j_1 \,j_2\,...\,j_n}{0\,i_1\,i_2\,...\,i_n}}Q_{\binom{j_1 \,j_2\,...\,j_n}{i_1\,i_2\,...\,i_n}}\begin{bmatrix}
    -\mathbf{y}^t\\
    -\mathbf{x}^t
\end{bmatrix}^{(j_0)}\begin{bmatrix}
   1 & (-1)^{i_{n}+j_n+1}\lambda
\end{bmatrix},\\
  W_n&= \sum_{\binom{j_0\,j_1 \,j_2\,...\,j_{n-1}\,j_n}{0\,i_1\,i_2\,...\,i_{n-1}\,0}} \delta_{\binom{j_0\,j_1 \,j_2\,...\,j_{n-1}\,j_n}{0\,i_1\,i_2\,...\,i_{n-1}\,0}}Q_{\binom{j_1 \,j_2\,...\,j_{n-1}}{i_1\,i_2\,...\,i_{n-1}}}\begin{bmatrix}
    -\mathbf{y}^t\\
    -\mathbf{x}^t
\end{bmatrix}^{(j_0)}\begin{bmatrix}
   \mathbf{x} & \mathbf{y}
\end{bmatrix}^{(j_{n})}.
\end{align*}
\end{lemma}


\begin{proof}
This follows easily from a direct computation using Lemma~\ref{lm:elementarycomp}.    
\end{proof}
\begin{lemma}[Recursion formula]\label{lm:Finftyrecursive}
     Notation as above. Let $X_n^+$ (resp. $Y_n^+$) be the sub-sum of $X_n$ (resp. $Y_n$) with $j_1$ even. For $n\geq 1$, define $\mathbf{U}_n$ to be the first row of $\begin{bmatrix}
     X_n^+& Y_n^+
     \end{bmatrix}$. 
 Then for $r>0$, we have
\begin{equation}
\label{eq: split U}
    \mathbf{U}_{n+r}= \sum_{\substack{\binom{j_1 \,j_2\,...\,j_r}{i_1\,i_2\,...\,i_r},\,2|j_1}}\delta_{\binom{j_1 \,j_2\,...\,j_r}{i_1\,i_2\,...\,i_r}}Q_{\binom{j_1 \,j_2\,...\,j_r}{i_1\,i_2\,...\,i_r}}\mathbf{U}^{(i_{r}+j_r+1)}_n.
\end{equation}
Moreover, the matrix $F_{n}$ can be recovered from $\mathbf{U}_n$ as follows: 
\begin{align}
    \begin{bmatrix}
    X_{n}&Y_{n}
\end{bmatrix}&= \begin{bmatrix}
    1\\
    \lambda^{-1}
\end{bmatrix}\mathbf{U}_n+ \begin{bmatrix}
    1\\
    -\lambda^{-1}
\end{bmatrix}\mathbf{U}_n^{(1)}, \label{eq: X+ Y+} \\
\begin{bmatrix}
    Z_{n}&W_{n}
\end{bmatrix}&=\sum_{j_0\geq 0}2 
\begin{bmatrix}
    -\mathbf{y}^t\\
    -\mathbf{x}^t
\end{bmatrix}^{(j_0)}
    \mathbf{U}^{(j_{0}+1)}_n. \label{eq: Z+ W+}
\end{align}

\end{lemma}
\begin{proof}
This follows from Lemma~\ref{lm:Finftyexplicit} and a direct computation. 
\end{proof}

\subsubsection{Deformation loci}
 For an element $v\in \cL=\mathrm{Span}_{\bZ_p}\{e', f', e_i, f_i\}$, we have $ \tilde{v}=F_\infty v\in \bL_P\otimes K[\![\underline{x},\underline{y}]\!]$.
There exists a rank $2m+2$ column vector $\mathbf{c}$ (resp. $\tilde{\mathbf{c}}$) valued in $\bZ_p$ (resp. $K[\![\underline{x},\underline{y}]\!]$) such that
\begin{equation}
\label{eqn: define c and c tilde}
v=(e',f';e_i,f_i)\mathbf{c},\;\;\tilde{v}=(e',f';e_i,f_i)\tilde{\mathbf{c}}.
\end{equation} 
Express $\mathbf{c}$ as $\sum_{k\geq 0} \mathbf{c}_kp^k$, where each entry of $\mathbf{c}_k$ is a Teichmüller lift of some element in $\mathbb{F}_p$. Let $\mathbf{U}_n$ be as in Lemma~\ref{lm:Finftyrecursive}. Define
\begin{equation}
\label{eqn: defined D_n's}
\begin{aligned}
     D_n(v)&:= 2\sum_{k\geq 0}\mathbf{U}_{n+k}\mathbf{c}_k \in W[\![\underline{x},\underline{y}]\!].\\
    \overline{D}_n(v)&:={D}_n(v)\bmod p\in k[\![\underline{x},\underline{y}]\!].
\end{aligned}  
\end{equation} 
\begin{remark}\label{rmk:D_np^m}
    Note that $D_n(p^rv)= D_{n+r}(v)$ by definition.
\end{remark}
\begin{proposition}\label{def locus sspecial}
For $n\geq 1$,  $x_{p^{n-1}v}=\overline{D}_{n}(v)$. 
\end{proposition}
We begin by a lemma: 
\begin{lemma}\label{lm;recursive2}
    The following are true: 
    \begin{enumerate}[label=\upshape{(\alph*)}]
    \item For $l\geq 0$, $D_n(v)^{(l)}-D_n(v)^{p^l}\in p^lW[\![\underline{x},\underline{y}]\!]$. 
        \item For $r>0$, we have: $$D_{a+r}(v)=\sum_{\substack{\binom{j_1 \,j_2\,...\,j_r}{i_1\,i_2\,...\,i_r},\,2|j_1}}\delta_{\binom{j_1 \,j_2\,...\,j_r}{i_1\,i_2\,...\,i_r}}Q_{\binom{j_1 \,j_2\,...\,j_r}{i_1\,i_2\,...\,i_r}}D_a(v)^{(i_{r}+j_r+1)}.$$ 
\item         The first two entries of $\tilde{\mathbf{c}}$ are 
$$\frac{1}{2}\sum_{n\geq 1}p^{-n}\left(\begin{bmatrix}
    1\\
    \lambda^{-1}
\end{bmatrix}D_{n}(v)+ \begin{bmatrix}
    1\\
    -\lambda^{-1}
\end{bmatrix}D_n(v)^{(1)}\right)+ O(1).$$
Here by $O(1)$ we mean that the entries have coefficients in $W$. The last $2m$ entries of $\tilde{\mathbf{c}}$ are 
$$\sum_{n\geq 1}p^{-n}\left(\sum_{j_0\geq 0} 
\begin{bmatrix}
    -\mathbf{y}^t\\
    -\mathbf{x}^t
\end{bmatrix}^{(j_0)}
   D_n(v)^{(j_{0}+1)}\right)+O(1).$$
   \item Let $v',w'$ be as in (\ref{eq:changeof basis}). If we write $\tilde{v}=(v',w'; e_i,f_i)\tilde{\mathbf{c}}'$, then the first two entries of $\tilde{\mathbf{c}}'$ are $ \sum_{n\geq 1}p^{-n} D_n(v)^{(1)}+O(1)$ and $ \sum_{n\geq 1}p^{-n} D_{n+1}(v)+O(1)$, respectively. The last $2m$ entries of $\tilde{\mathbf{c}}'$ coincide with that of $\tilde{\mathbf{c}}$. 
    \end{enumerate}
\end{lemma}

\begin{proof}
Note that $a^{(l)}=a^{p^l}$ for a Teichmüller lift. For part (a),  we can reduce to showing that if $u$ is an entry of $2\mathbf{U}_{n}$, then $u^{(l)}\equiv u^{p^l}(\bmod p^l)$. This follows from the explicit formulae in Lemma~\ref{lm:Finftyexplicit}. Part (b) and part (c) are direct consequences of Lemma~\ref{lm:Finftyrecursive}. Part (d) follows from (c) and the change of basis $w' = \frac{1}{2\lambda p}(\lambda e' + f'), v' = \frac{1}{2\lambda}(\lambda e' - f')$.
\end{proof} 

\begin{proof}[Proof of Proposition~\ref{def locus sspecial}] Let $v$ be any formal special endomorphism, it suffices to show that  $x_v=\overline{D}_1(v)$. It suffices to show that for any map $\iota: \Spf k[\![t]\!]\rightarrow \Spf k[\![\underline{x},\underline{y}]\!]$, $v$ deforms to $\Spf(k[\![t]\!]/t^{\alpha})$ but not $\Spf(k[\![t]\!]/t^{\alpha+1})$, where $\alpha$ is the degree of the series $\iota^*\overline{D}_1(v)\in k[\![t]\!]$. By Lemma~\ref{T:PDlemma}, it suffices to check that  $\tilde{\iota}^* \tilde{\mathbf{c}}'\in  W[\![t,\frac{t^{\alpha p}}{p}]\!]^{\oplus (2m+2)}$ but not $ W[\![t,\frac{t^{(\alpha+1) p}}{p}]\!]^{\oplus (2m+2)}$, where $\tilde{\mathbf{c}}'$ is in Lemma~\ref{lm;recursive2}(d), and $\tilde{\iota}:\Spf W[\![t]\!]\rightarrow \Spf W[\![\underline{x},\underline{y}]\!]$ is a lift of $\iota$.  

Let $n\geq 2$. Taking $a=1$ and $r=n-1$ in Lemma~\ref{lm;recursive2} (b), combining Lemma~\ref{lm;recursive2} (a), and noting that $i_{n-1}\geq 0$ and $j_{n-1}\geq n-2$, we have: 
\begin{equation}\label{eq:3434}
D_{n}(v)^{(1)}=\sum_{\substack{\binom{j_1 \,j_2\,...\,j_{n-1}}{i_1\,i_2\,...\,i_{n-1}},\,2|j_1}}\delta_{\binom{j_1 \,j_2\,...\,j_{n-1}}{i_1\,i_2\,...\,i_{n-1}}}Q^{(1)}_{\binom{j_1 \,j_2\,...\,j_{n-1}}{i_1\,i_2\,...\,i_{n-1}}}D_1(v)^{p^{n}}+O(p^{n}).   
\end{equation}
If $\overline{D}_1(v)\in (t^{\alpha})$, then $\tilde{\iota}^*D_1(v)\in (p,t^\alpha)$. Taking $\tilde{\iota}^*$ on both sides of (\ref{eq:3434}), and plugging in Lemma~\ref{lm;recursive2}(d), we find that the first entries of $\tilde{\iota}^*\tilde{\mathbf{c}}'$ lies in $ W[\![t,\frac{t^{\alpha p}}{p}]\!]$. A similar argument shows that other entries of $\tilde{\iota}^*\tilde{\mathbf{c}}'$ also lie in $ W[\![t,\frac{t^{\alpha p}}{p}]\!]$. Furthermore, taking $n=1$ in (\ref{eq:3434}) gives $D_1(v)^{(1)}=ct^{\alpha p}+O(p)$ for some constant $c$, one easily see that the first entries of $\tilde{\iota}^*\tilde{\mathbf{c}}'$  lies out side of $W[\![t,\frac{t^{(\alpha+1) p}}{p}]\!]$.  
\end{proof}

\subsubsection{Proof of Theorem~\ref{thm:eventually_rapid_decay} for supersingular points} We will only work out the most difficult case where $P$ is superspecial. Suppose that the equation of the formal curve is given by \begin{equation}\label{eq:localequation1}
    \cC:t\rightarrow (x_1(t),...,x_m(t);y_1(t),...,y_m(t)),\;\;x_i(t),y_i(t)\in k[\![t]\!].
\end{equation}
\begin{notation}\label{not:tech_curve}
We will also use  $x_i(t), y_i(t)$ to denote the power series in  $W[\![t]\!]$ where all the coefficients are Teichmuller lifts of the coefficients of $x_i(t), y_i(t)$ in (\ref{eq:localequation1}). This abuse of notation won't cause confusion if one keeps track of which base ring they work in. To ease notation, we often just write $x_i,y_i$ for the corresponding power series in $t$.
\end{notation}
Let $h_0$ be the intersection number of $\cC$ with the non-ordinary locus, and note that $h_0=v_t(Q_0(t))$ by \cite[Lemma 4.9]{MST22}. This is a positive integer. We also set $h_a:= v_t( Q_a(t))$ for $a>0$. Write ${D}_n(v)(t)$ for ${D}_n(v)$ evaluated at $x_i(t),y_i(t)$. Write $d(v):=v_t\left( \overline{D}_1(v)(t)\right)$ for any $v\in \cL$. Proposition~\ref{def locus sspecial} implies that $v$ lifts to $\mathrm{Spf}\,k[\![t]\!]/(t^{d(v)})$ but not beyond.  

\begin{claim}
 Let $v\in \cL$ be such that $d(v)<\infty$, then $d(pv)\geq pd(v)$. If furthermore $d(v)\geq  h_0$, then 
 $d(pv)= h_0+pd(v)$.
\end{claim}  It follows from 
  Remark~\ref{rmk:D_np^m} and   Lemma~\ref{lm;recursive2}(b) for  $n=r=1$ that 
  $$D_{1}(pv)=\sum_{\substack{\binom{j_1 }{i_1},\,2|j_1}}\delta_{\binom{j_1 }{i_1}}Q_{\binom{j_1 }{i_1}}D_1(v)^{(i_{1}+j_1+1)}.$$
Since $i_1+j_1+1\geq 1$, we find that $d(pv)\geq pd(v)$. If furthermore $d(v)\geq  h_0$, then the lowest $v_t$-value among the set $\left\{\delta_{\binom{j_1 }{i_1}}Q_{\binom{j_1 }{i_1}}D_1(v)^{(i_{1}+j_1+1)}\bmod p\right\}$ is achieved when $i_1=j_1=0$ (say, since $d(v)\geq h_0$, we must have $h_0+pd(v)<h_i+p^{i+1}d(v)$ for all $i\geq 1$). This proves the \textit{Claim}. 

Let $\Gamma\subseteq \cL$ be a saturated sublattice that eventually decays. Let $N:=\lfloor\log_p h_0\rfloor+1$. By the first assertion in the \textit{Claim}, we have $d(p^Nv)\geq h_0$ for every primitive $v\in \Gamma$. Now we claim that the set $$\cP=\{d(p^Nv)|v\in \Gamma \text{ is primitive}\}$$  
is bounded. If this was not the case, then there is a sequence of primitive vectors $\{v_k\}_{k\in\bN}\subseteq \Gamma$ such that $d(p^Nv_k)\geq k$. Let $v\in \Gamma$ be a limit point of $\{v_k\}_{k\in\bN}$ in the $p$-adic metric, which exists because of the compactness of $\Gamma$. Then $v$ is primitive and $p^Nv$ lifts to $\cC$. This contradicts the assumption that $\Gamma$ eventually decays. 

Let $D$ be an integer upper bound for $\cP$. Then for every primitive $v\in \cL$, we have $D\geq d(p^Nv)\geq h_0$. Iterating the second assertion of \textit{Claim}, we see that for all $n\geq 0$, $$d(p^{N+n}v)=(1+p+...+p^{n-1})h_0+p^nd(p^Nv)\leq (1+p+...+p^{n-1})h_0+p^nD.$$
It is then easy to see that for all $n\geq 0$, 
$$d(p^{n}v)\leq (1+p+...+p^{n})D.$$
Therefore $p^nv$ lifts at most to $\Spf k[\![t]\!]/(t^{(1+p+...+p^n)D})$. This establishes the theorem in the supersingular case. $\hfill\square$
 
\subsection{Non-supersingular case}
Suppose that $P$ is non-supersingular of height $h$. The case where $P$ is ordinary admits a very easy proof by canonical coordinates. Indeed, $\cS_W^{/P}$ canonically admits a structure of formal torus. Let $v\in \cL$ be a primitive element, then $x_v$ cuts out a formal subtorus $T_v \subseteq S^{/P}$, and $x_{p^nv}=x_v^{p^n}$. Together with the compactness of $\Gamma$, this immediately establishes   Theorem~\ref{thm:eventually_rapid_decay} in the case where $P$ is ordinary. 

When $P$ is not ordinary, we don't have canonical coordinates, so we will resort to a strategy similar to the supersingular case. However, forthcoming work of Apoorva Aggarwal and Dan Townsend \cite{AT} establishes an analogue of Serre-Tate coordinates on non-ordinary non-supersingular strata of GSpin Shimura varieties, and this could lead to a more conceptual proof. 

In the following, we use the explicit coordinates deduced from \S\ref{sub:nons_setup}.

\subsubsection{Explicit computation for $R_\infty$}\label{sub:R_infty}
Let $v\in \cL=\mathrm{Span}_{\bZ_p}\{e_i, f_i\}_{i\geq h}$. 

Then $\tilde{v}$ can be computed explicitly as follows: Denote by $R_n$ the right  $2(m+1)\times 2(m-h+1)$ block of the iterated Frobenius 
 $\Frob^n$. Then the action of $\lim_{n\rightarrow\infty}\Frob^n$ on $\cL$ is presented by the matrix $R_\infty=\lim_{n\rightarrow\infty}R_n$, which has entries in $K[\![\underline{x},\underline{y}]\!]$.

\begin{proposition}\label{prop:Shape of R_infty} 
 Write $R^\blacktriangle_\infty$ (resp. $R^\blacktriangledown_\infty$) for the block of top $2h$ (resp. bottom $2(m-h+1)$) rows of $R_\infty$.  Then
\begin{enumerate}
    \item The even rows of $R_\infty^\blacktriangle$ are $\mathbf{0}$.
    \item $R_\infty^\blacktriangledown= I_{2(m-h+1)}$.
    \item Define a $\bZ_p$-linear operator $\theta$ on $R$ by $\theta(w)=\sum_{j=0}^{h-1} x_jw^{(h-j)}$ (where $x_0=1$ by convention), and let $\tau_{\theta}(w)=\sum_{n\geq 0}\frac{\theta^n(w)}{p^n}$. Then the first row of $R_\infty^\blacktriangle$ is $$\mathbf{r} = [\tau_\theta(x_1),\tau_\theta(y_1),...,\tau_\theta(x_m),\tau_\theta(y_m)].$$
For $j=2,3,...,h$, the $(2j-1)$-th row of $R_\infty^\blacktriangle$ is $p^{-1}\mathbf{r}^{(h-j+1)}$.

\end{enumerate}
    
\end{proposition}
\begin{proof}
Write $\Frob$ as  (\ref{eq:Frob_parts}). Then elementary computation shows that \begin{align}
\label{eq:R_inf 1} R^\blacktriangle_\infty&=\sum_{i\geq 0}AA^{(1)}...A^{(i-1)}B^{(i)},\\ \label{eq:R_inf 2} 
 R^\blacktriangledown_\infty&=I_{2(m-h+1)}+\sum_{i\geq 0} D^{(i)} (R^\blacktriangle_\infty)^{(i+1)}.\end{align}
Part (1) follows directly from (\ref{eq:R_inf 1}) and part (2) follows from (\ref{eq:R_inf 2}). Part (3) can also be obtained from computation: 

Let $\mathbf{c}:=[a_h,b_h,\cdots,a_m,b_m]^t$ be a column vector of $p$-adic integers. Define $\mathbf{w}:=B\mathbf{c}$, $\widetilde{\mathbf{w}}:= R_\infty^\blacktriangle\mathbf{c}$. Then $R^\blacktriangle_\infty$ can be recovered from $\widetilde{\mathbf{w}}$ by taking $\mathbf{c}$ to be various unit vectors. Observe the following recursive relation \begin{equation}\label{eq:recurve_w}\widetilde{\mathbf{w}}=\mathbf{w}+A\widetilde{\mathbf{w}}^{(1)}.\end{equation}
Now (1) implies that $\widetilde{\mathbf{w}}=[f_1,0,...,f_h,0]^t$ with $f_i\in K[\![x_1,...,x_m,y_1,...,y_m]\!]$. Take the convention that $x_0=1$ and let $z=[x_h,y_h,...,x_m,y_m]\mathbf{c}$. One reads from (\ref{eq:recurve_w}) that, for all $h\geq 1$, 
\begin{align*}
    f_1&=p^{-1}x_{h-1}f_1^{(1)}+x_0f_2^{(1)}+x_1f_3^{(1)}+x_2f_4^{(1)}+...+x_{h-2}f_h^{(1)}+z,\\
    f_j&=f_{j+1}^{(1)},\;\;j=2,3,...,h-1,\\
f_h&=p^{-1}f_1^{(1)}.
\end{align*}
It follows from the second and third formulae that \begin{equation}\label{eq:f_j in f_1}
f_j=p^{-1}f_1^{(h-j+1)}
\end{equation} for all $2\leq j\leq h$. Substituting this to the first formula, we get an equation
\begin{equation}\label{eq:f_1_iterate}
f_1=p^{-1}\sum_{j=0}^{h-1} x_jf_1^{(h-j)}+ z,
\end{equation}
whose solution is 
$$f_1=\sum_{n\geq 0} \frac{\theta^n(z)}{p^n}.$$
Letting $\mathbf{c}$ be various unit vectors to get the first row of $R_\infty^\blacktriangle$. Plug $f_1$ into (\ref{eq:f_j in f_1}) to get other odd rows. 
\end{proof}

\subsubsection{Deformation loci}\label{subsub:deflocinonss} 
For an element $v\in \cL$, there exists a column vector  
$\mathbf{c}=[a_h,b_h,\cdots,a_m,b_m]^t$ of $p$-adic integers with $v=[e_h,f_h,...,e_m,f_m]\mathbf{c}$. Let  $\tilde{\mathbf{c}}=R_{\infty}\mathbf{c}$, then $\tilde{v}=[e_1,f_1,...,e_m,f_m]\tilde{\mathbf{c}}$. Take Teichmüller expansion $\mathbf{c}=\sum_{k\geq 0} \mathbf{c}_kp^k$. Let $\theta$ be as per Proposition~\ref{prop:Shape of R_infty}. For $n\geq 0$, define 
\begin{equation}
\label{eqn: defined D_n's2}
\begin{aligned}
  D_n(v)&= \sum_{k\geq 0} \theta^{n+k}([x_h,y_h,...,x_m,y_m]\mathbf{c}_k) 
     \in W[\![\underline{x},\underline{y}]\!].\\
    \overline{D}_n(v)&:={D}_n(v)\bmod p\in k[\![\underline{x},\underline{y}]\!].
\end{aligned} 
\end{equation}
Note that $D_n(p^rv)=D_{n+r}(v)$. The main result of this section is 
\begin{proposition}\label{def locus nonss}
For $n\geq 0$,  $x_{p^{n}v}=\overline{D}_{n}(v)$. 
\end{proposition}
We begin by a lemma in the same style of Lemma~\ref{lm;recursive2}.
\begin{lemma}\label{lm:descDn_nonss} The following are true:
    \begin{enumerate}[label=\upshape{(\alph*)}]
        \item For $l\geq 0$, $D_n(v)^{(l)}-D_n(v)^{p^l}\in p^l W[\![\underline{x},\underline{y}]\!]$.
    \item For $r>0$, we have $$D_{n+r}(v)= \sum_{i_1,i_2,...,i_r\in \{0,1,...,h-1\}}x_{i_1}x_{i_2}^{(h-i_1)}x_{i_3}^{(2h-i_1-i_2)}...x_{i_{r}}^{((r-1)h-i_1-i_2-...-i_{r-1})}D_n(v)^{(rh-i_1-i_2-...-i_r)}.$$ 
    \item  The first entry of $\tilde{\mathbf{c}}$ is 
    $\sum_{n\geq 0}p^{-n}D_n(v)$. For $j=2,3,...,h$, the $(2j-1)$-th entry of $\tilde{\mathbf{c}}$ is  $\sum_{n\geq 1}p^{-n}D_{n-1}(v)^{(h-j+1)}$. All other entries of $\tilde{\mathbf{c}}$ lie in $\bZ_p$. 
    \end{enumerate}
\end{lemma}
\begin{proof}
   Part (a) is similar to Lemma~\ref{lm;recursive2}. Part (b) follows from an explicit expansion of $\theta^n$: $$\theta^n(f)=\sum_{i_1,i_2,...,i_n\in \{0,1,...,h-1\}}x_{i_1}x_{i_2}^{(h-i_1)}x_{i_3}^{(2h-i_1-i_2)}...x_{i_{n}}^{((n-1)h-i_1-i_2-...-i_{n-1})}f^{(nh-i_1-i_2-...-i_n)},$$
   which can be easily shown by induction. Part (c) follows from Proposition~\ref{prop:Shape of R_infty}.
\end{proof}
\begin{proof}[Proof of Proposition~\ref{def locus nonss}] The proof is similar to Proposition~\ref{def locus sspecial}. Let $v$ be any formal special endomorphism, it suffices to show the deformation locus of $v$ is cut out by $\overline{D}_0(v)$. This reduces to checking that for any map $\iota: \Spf k[\![t]\!]\rightarrow \Spf k[\![\underline{x},\underline{y}]\!]$, $v$ deforms to $\Spf(k[\![t]\!]/t^{\alpha})$ but not $\Spf(k[\![t]\!]/t^{\alpha+1})$, where $\alpha$ is the degree of the series $\iota^*\overline{D}_0(v)\in k[\![t]\!]$. By Lemma~\ref{T:PDlemma}, it suffices to check that $\tilde{\iota}^* \tilde{\mathbf{c}}\in  W[\![t,\frac{t^{\alpha p}}{p}]\!]^{\oplus (2m+2)}$ but not $ W[\![t,\frac{t^{(\alpha+1) p}}{p}]\!]^{\oplus (2m+2)}$. For this, we use Lemma~\ref{lm:descDn_nonss}. The detail is left to the reader.
\end{proof}
\subsubsection{Proof of Theorem~\ref{thm:eventually_rapid_decay} for non-supersingular points} \label{sub:non-sscase}  Suppose that the equation of the formal curve is given by \begin{equation}\label{eqq:localequation1}
    \cC:t\rightarrow (x_1(t),y_1(t),...,x_m(t),y_m(t)),\;\;x_i(t),y_i(t)\in k[\![t]\!].
\end{equation}
Let the convention be the same as Notation~\ref{not:tech_curve}.

Let $h_0$ be the intersection number of $\cC$ with the non-ordinary locus. Using a similar argument as \cite[Lemma 4.9]{MST22}, we have $h_0=0$ when $h=1$, and $h_0=v_t(x_{h-1})>0$ when $h\geq 2$. Write ${D}_n(v)(t)$ for ${D}_n(v)$ evaluated at $x_i(t),y_i(t)$. For  $v\in \cL$, we write $d(v):=v_t\left( \overline{D}_0(v)(t)\right)$. Proposition~\ref{def locus nonss} implies that $v$ lifts to $\mathrm{Spf}\,k[\![t]\!]/(t^{d(v)})$ but not beyond.  

It follows from 
   Lemma~\ref{lm:descDn_nonss}(b) for  $n=0,r=1$ that 
  $$D_{0}(pv)=x_0D_0(v)^{(h)}+x_1D_0(v)^{(h-1)}+...+x_{h-1}D_0(v)^{(1)}.$$
When $h=1$, this simplifies to $D_0(pv)=D_0(v)^{(1)}$, and the theorem is easy. Now we suppose that $h\geq 2$. 
\begin{claim}
    Let $v\in \cL$ be such that $d(v)<\infty$, then $d(pv)\geq pd(v)$. If furthermore $ d(v)\geq h_0$, then $d(pv)= h_0+pd(v)$.
\end{claim}  
Again the first assertion is easy by the above formula for $D_0(pv)$. If $ d(v)\geq h_0$, then it is easy to see that the lowest $v_t$-value among the set $\left\{x_iD_0(v)^{(h-i)}\bmod p\right\}$ is achieved when $i=h-1$. This proves the \textit{Claim}. 

The rest is similar to the supersingular case.$\hfill\square$
\section{Uniform decay at the boundary} 
In this section we study the decay behavior of special endomorphisms of log 1-motives on the boundary. 
Let $\mathscr{S}^{\Sigma}$ be a toroidal compactification of $\mathscr{S}$ as discussed in \S\ref{sub:toroidalSpin}, viewed as a log scheme with log structure coming from the boundary. \begin{setup}\label{setp:logmorphi}
Let $C$ be a proper smooth connected  curve with a non-constant map $\iota:C\rightarrow \mathscr{S}_{k}^\Sigma$ whose image lies generically in the ordinary locus of $\mathscr{S}_{k}$. Let $P\in  C(k)$ be a boundary point that maps to a stratum labeled by $(\Phi,\sigma)$ such that $\sigma$ is an one dimensional inner ray\footnote{Note that all type II boundary points have this property.}, we will use $o$ to denote its image in $\mathscr{S}^\Sigma(k)$. Let $t$ be a uniformizer at $P$.  Then we get a map of formal schemes 
$\hat{\iota}:\Spf k[\![t]\!]\rightarrow \mathscr{S}^{\Sigma,/o}_{k}$. Endow $\Spf k[\![t]\!]$ with the log structure coming from the closed point, then $\hat{\iota}$ upgrades to a morphism of log formal schemes. 
\end{setup}


Let $\cQ_o$ be the log 1-motive over $o$, with log structure restricted from the embedding $o\hookrightarrow \mathscr{S}^{\Sigma}$; cf. \S\ref{sub:locstructureatboundarypt}. There is a notion of special endomorphisms for $\cQ_o$, see Definition~\ref{def:specialendoboundary}, that behaves very much similar to the classical special endomorphisms for abelian varieties. The set  $L(\cQ_o)$ of special endomorphisms is a  lattice equipped with a positive definite quadratic pairing $\la -,- \ra$ induced by self-composition; this is a consequence of Corollary~\ref{prop:spcialendolatticeboundary}.  

The main results of this section is: 
\begin{theorem}\label{thm:globalanalysis}
In Setup~\ref{setp:logmorphi}, assume that $k=\Fpbar$. Then there is a uniform bound $N$ such that $v$ does not lift to $\Spf k[\![t]\!]/(t^N)$ for any vector $v\in L(\cQ_o)$ with $p\nmid \la v, v\ra$. 
\end{theorem}
This will be proved at the end of this section. 

\subsection{Log 1-motives} This section is a brief review of log 1-motives. We assume that the reader is familiar with basic notions in log geometry; cf. \cite{Kk89}. To set up the notation, we will use triple $(S,\mathrm{M}_S,\alpha)$ to denote a log scheme (or log formal scheme), where $S$ is a scheme (formal scheme), $\mathrm{M}_S$ is a sheaf of monoids over $S_{\et}$, and a morphism $\alpha: \mathrm{M}\rightarrow \mathcal{O}_S$. We will usually denote $(S,\mathrm{M}_S,\alpha)$ simply by $(S,\mathrm{M}_S)$ or even $S$ if the log structure is clear from the context. If $S$ is a regular scheme with a reduced normal crossing divisor $D$, then one can assign $S$ a log structure that comes from $D$; cf. \cite[\S 1.5(1)]{Kk89}.

\begin{ass}
    If not otherwise specified, all log schemes in this paper will be assumed to be noetherian with fs log structure. 
\end{ass}
\begin{notation}
A monoid $\cP$ is called sharp, if $  1$ is the unique invertible element. For any sharp and fs monoid $\cP$, and any ring $R$, we can associate the log ring $R_{\cP}$, with log structure $\mathrm{M}_S=\bG_m\times \cP$ and $\alpha: \mathrm{M}_S\rightarrow R$ is the map taking $\cP\setminus\{1\}\rightarrow 0$.  For example, $W_{\bN^r}$ is the ring $W$ with log structure $\alpha:W^*\times \bN^r\rightarrow W$ taking all non-identity elements of $\bN^r$ to 0. 
\end{notation}

A \textbf{log 1-motive} (\cite[Definition 2.2]{LogAV} of type $(Y,X)$ over $(S,\mathrm{M}_S)$ is a complex $Q=[Y\xrightarrow{u} J^{\log}]$ in degrees -1 and 0, where $Y$ is an isotrivial $\bZ$-lattice (i.e., locally constant in $S_\et$), $J$ is a semi-abelian scheme over $S$ which is an extension of an abelian scheme $B$ by an isotrivial torus $T$ with character lattice $X$, and $J^{\log}$ is a certain enlargement of $J$ in the Kummer log flat topos, defined as follows: 
First, the Kato's log torus $\bG_m^{\log}$, which is a sheaf in Kummer log flat topos, is defined by the functor
$ (S',N)\rightarrow \Gamma(S',N^{\mathrm{gp}})$. 
There is a natural embedding $\bG_m\hookrightarrow\bG_m^{\log}$. We define $J^{\log}$ be the pushout of $J$ along the embedding $T\hookrightarrow T^{\log}=\Hom(X,\bG_m^{\log})$.

\subsubsection{Morphisms}
Let $G,H$ be semi-abelian schemes over $S$. For $\beta\in \Hom(G,H)$, there is a natural induced map $\beta^{\log}\in\Hom(G^{\log},H^{\log})$. 
\begin{lemma}[{\cite[Proposition 2.5]{LogAV}}]\label{lm:logidenticaliso}
Notation as above. The natural map $\Hom(G,H)\rightarrow \Hom(G^{\log},H^{\log})$ is an isomorphism.    
\end{lemma}
Now let $Q_1=[Y_1\rightarrow J_1^{\log}]$ and $Q_2=[Y_2\rightarrow J_2^{\log}]$ be two log 1-motives over $S$.  
A morphism $Q_1\rightarrow Q_2$ is a commutative square 
$$
\begin{tikzcd}
Y_1 \arrow[r] \arrow[d, "A"'] & J_1^{\log} \arrow[d, "B"] \\
Y_2 \arrow[r]                      & J_2^{\log}                         
\end{tikzcd}$$
By Lemma~\ref{lm:logidenticaliso}, $B=\beta^{\log}$ for some $\beta\in\Hom(J_1,J_2)$. We will denote a morphism as a pair $(A,\beta^{\log})$.

\subsubsection{Dual and polarization} See \cite[\S 2]{LogAV}. We just note that if $Q=[Y\rightarrow J^{\log}]$ is of type  $(Y,X)$, then $Q^\vee=[X\rightarrow (J^\vee)^{\log}]$ is of type $(X,Y)$. And a polarization $\lambda: Q\rightarrow Q^\vee$ is a morphism $(\lambda^\et, (\lambda^{\mathrm{sab}})^{\log})$ where $\lambda^{\et}:Y\rightarrow X$ is injective with finite cokernel, and $\lambda^{\mathrm{sab}}:J\rightarrow J^\vee$ is a certain isogeny that satisfies some extra properties that we don't specify.

\subsubsection{Log $p$-divisible groups}Following \cite{logDiu}, we have a category $(\mathrm{fin}/S)_d$ of log finite groups schemes stable under the Cartier {d}uality. By \cite[Proposition 1.4]{logDiu}, log finite group schemes are (classically) finite, log flat, fs and Kummer type log schemes over $S$. Following \cite[\S 4.1]{logDiu}, we can define the category $(p\mathrm{-div}/S)_{d}$ as the full subcategory of the category of sheaves of
abelian groups on $S$ consisting of objects $\mathscr{G}$ satisfying (a) $\mathscr{G}=\bigcup_{n} \mathscr{G}([p^n])$, (b)  $p:\mathscr{G}\rightarrow \mathscr{G}$ is surjective, and (c) For each $n$, $\mathscr{G}([p^n])$ belongs to $(\mathrm{fin}/S)_{d}$.
We will call an object in $(\mathrm{fin}/S)_{d}$ a \textbf{log $p$-divisible group}. 

One can functorially associate a log $p$-divisible group $Q[p^{\infty}]\in (p\mathrm{-div}/S)_{d}$ to a log 1-motive $Q=[Y\xrightarrow{u} J^{\log}]$; see \cite[Proposition 3.5, Definition 3.6]{Sheer}.

\subsubsection{Crystalline realization} Equip $W$ with trivial log structure. Let $(S,\mathrm{M})$ be a log scheme over $W$. There is a notion of \textbf{log crystalline site} $(S/W)_{\cris}^{\log}$ and \textbf{log crystals} introduced by Kato in \cite[\S5,\S6]{Kk89}. 
When $(S,\mathrm{M})$ is sufficiently nice, a log crystal can also be described as a filtered module with log connection; cf. \cite[Theorem 6.2]{Kk89}. This language is used in \cite[\S 1.3]{MP11}. Following \cite[\S  1.3.3]{MP19}, a \textbf{log Dieudonné crystal} on an fs and sharp log scheme $(S,\mathrm{M})$ is a four tuple $(\mathbb{M},F,V,\Fil^\bullet)$, where $\mathbb{M}$ is a crystal on $(S/W)_{\cris}^{\log}$, $\varphi$ is  Frobenius, $V$ is Verschiebung and $\Fil^\bullet$ is a two step Hodge filtration. For a log 1-motive $Q$ over a locally $p$-nilpotent fs sharp log scheme, one can functorially attach a Dieudonné crystal $\bD(Q)$.

\subsubsection{Algebraization}
Let $(R,\mathfrak{m})$ be a complete local Noetherian ring with fs log structure. Note that this implies the existence of a finite integral chart $\cP\rightarrow R$.  

\begin{lemma}\label{eq:existencetheorem}Notation as above. 
\begin{enumerate}
    \item  The category of log 1-motives over $\Spec R$ is equivalent to the category of log 1-motives over $\Spf R$.
\item The category $(p-\mathrm{div}/\Spec R)_d$ is equivalent to $(p-\mathrm{div}/\Spf R)_d=
\lim(p-\mathrm{div}/\Spec R/\mathfrak{m}^k)_d$.
\end{enumerate}    \end{lemma}\begin{proof}\begin{enumerate}
    \item Suppose there is an compatible system of log 1-motives $Q_n=[Y\rightarrow J_n^{\log}]$ over $\Spec R/\mathfrak{m}^n$. By Grothendieck's    existence theorem, there is a $J$ over $\Spec R$ reducing to $J_n$ modulo $\mathfrak{m}^n$. To show $J^{\log}$ reduces to $J_n^{\log}$ modulo $\mathfrak{m}^n$, it suffices to show that the log structure on $R$ is equal to the inverse limit of log structures on $R/\mathrm{m}^n$ for $n\in \bN$. This follows from the existence of a finite integral chart $\cP\rightarrow R$. 
    \item 
    By definition, an object $\mathscr{G}\in(p\mathrm{-div}/S)_{d}$ is a union of log finite group schemes $\mathscr{G}([p^n])$. It suffices to show that if $H$ is a log finite group scheme over $\Spf R$, then we can algebraize it to a log finite group scheme over $\Spec R$. The underlying finite scheme structure of $H$ algebraizes by formal GAGA. The log structure also algebraizes by the  existence of a finite integral chart $\cP\rightarrow R$. 
\end{enumerate} \end{proof}

\subsubsection{The classical part}\label{subsub:Tp} Let $(R,\mathfrak{m})$ be a complete local Noetherian ring with log structure associated to a finite integral chart $\cP\rightarrow \mathfrak{m}$. We fix a splitting $\mathrm{M}^{\mathrm{gp}}_{R}=R^*\times \cP^{\mathrm{gp}}$. 

\begin{defn}
  Consider a log 1-motive $Q=[Y\xrightarrow{u} J^{\log}]$ of type $(Y,X)$ over $R$. The \textbf{classical part} of $Q$  is a classical 1-motive $Q^{\mathrm{cl}}=[Y\xrightarrow{u^{\mathrm{cl}}} J]$ defined by 
    $$u^{\mathrm{cl}}:Y\xrightarrow{u } J^{\log}(R)=J(R)\times (X^\vee\otimes \cP^{\mathrm{gp}})\xrightarrow{\mathrm{proj}}J(R).$$
 Here the splitting of $J^{\log}(R)$ is induced from the fixed splitting of $\mathrm{M}^{\mathrm{gp}}_{R}$. Subtracting $u^{\mathrm{cl}}$ from $u$, we also get a (toric) log 1-motive $Q^{\mathrm{tor}}=[Y\xrightarrow{u^{\mathrm{tor}}} T^{\log}]$ with the property that $u=u^{\mathrm{cl}}+u^{\mathrm{tor}}$.   
\end{defn}


\begin{lemma}\label{lm:vtocl}Let $v=(A,\beta^{\log})\in \End(Q)$. Then $v^{\mathrm{cl}}:=(A,\beta)\in \End(Q^{\mathrm{cl}})$ and  $v^{\mathrm{tor}}:=(A,\beta^{\log})\in \End(Q^{\tor})$. As a result, $\End(Q)=\{(A,\beta)\in \End(Q^{\mathrm{cl}}): (A,\beta^{\log})\in \End(Q^{\mathrm{tor}})\}$.
\end{lemma} 
\begin{proof}
Suffices to show that $v^{\mathrm{cl}}\in \End(Q^{\mathrm{cl}})$. The graph $\Gamma_v\subseteq Q\times Q$ is a log 1-motive, whose classical part $\Gamma^{\mathrm{cl}}_v\subseteq Q^{\mathrm{cl}}\times Q^{\mathrm{cl}}$ is the graph of an endomorphism which coincides with $v^{\mathrm{cl}}$. 
\end{proof}

Let $o$ be the closed point of $\Spec R$, with the induced log structure. Note that Lemma~\ref{lm:vtocl} applies to $o$ as well. 

\begin{lemma}\label{lm:extlogII} Let $j:\Spec R'\rightarrow \Spec R$ be a morphism of local log schemes. Let $v_o=(A,\beta^{\log}_o)\in \End(Q_o)$ and $v_o^{\mathrm{cl}}=(A,\beta_o)\in \End(Q_o^{\mathrm{cl}})$. Then $v_o$ extends to an endomorphism $(A,\beta^{\log})\in \End(j^*Q)$\footnote{This means that the image of $v_o$ in $\End(Q_{o'})$ lies in the image of $\End(j^*Q)\rightarrow \End(Q_{o'})$. Here $o'$ is the closed point of $\Spec R'$ with the induced log structure.} if and only if $v_o^{\mathrm{cl}}$ extends to an endomorphism $(A,\beta)\in \End(j^*Q^{\mathrm{cl}})$. 
\end{lemma}
\begin{proof}
Since $u^{\tor}(Y)\subseteq \{1\}\times (X^\vee\otimes \cP^{\mathrm{gp}})\subseteq J^{\log}(R)$, we see that $v_o^{\tor}\in \End(Q^{\tor}_o)$ extends to an endomorphism $v_{R}^{\tor}\in \End(Q^{\tor})$, hence an endomorphism $v^{\tor}_{R'}\in \End(j^*Q^{\tor})$. The rest is easy given the above observation.\end{proof}

\subsection{Local structure at a boundary point}\label{sub:locstructureatboundarypt} Let $\iota:\mathscr{S}^\Sigma\rightarrow \mathscr{A}_g^{\Sigma^\ddagger}$ be as in \S\ref{sub:toroidalSpin}.
 Possibly refining $\Sigma^\ddagger$, we can assume that a polyhedral cone $\sigma$ for $\Sigma$ is also a polyhedral cone for $\Sigma^\ddagger$. Fix a point $o\in \mathscr{S}^\Sigma(k)$ that lies in a stratum labeled by $(\Phi,\sigma)$.  Then $\iota(o)\in\mathscr{A}_g^{\Sigma^\ddagger}(k)$ lies in the strata labeled by $(\iota_*\Phi,\sigma)$. Let $\mathscr{A}_{g,W}^{\Sigma^\ddagger,/\iota(o)}=\Spf R$ and  $\mathscr{S}^{\Sigma,/o}_W=\Spf R_G$. The ring $R$ is equipped with a log structures coming from the boundary divisor, similar for $R_G$. And we have a continuous map of log rings $\iota^\sharp:R\rightarrow R_G$ which is the normalization of a surjection; cf. \cite{MP11}. There is an $r\in \bN$ such that the log structure on $R$ (resp. $R_G$) admits a chart $\bN^r\rightarrow R$ (resp. $\bN^r\rightarrow R_G$). Note that the log point $o\in \Spec R_G$ and $\iota(o)\in \Spec R$ are both $\Spec k_{\bN^r}$. Therefore we can identify them. In the following, we will simply denote $\iota(o)$ also by $o$.  

The degenerating abelian scheme over $\Spec R$ corresponds to a log 1-motive $\cQ$ over $\Spec R$ with principal polarization $\lambda$; cf. \cite[Proposition 1.2.4.2]{MP11}. Let $\mathcal{Q}_o$ be the special fiber of $\mathcal{Q}$ at the log point $o$. Then $\Spf R$ is precisely the universal deformation space of $(\mathcal{Q}_o,\lambda_o)$ (cf. \cite[Theorem 4.2.1.3(3)]{MP11}), and $(\mathcal{Q},\lambda)\otimes \Spf R$ is its universal deformation. 

 Associated to $\mathcal{Q}$ is a polarized log Dieudonné crystal $\mathbb{D}(\mathcal{Q})$. One can use Tate tensors over $\mathbb{D}(\mathcal{Q}_o)$ to obtain an explicit model for the map $\iota^\sharp:R\rightarrow R_G$, which we now describe. Let $M_o$ be the evaluation of $\bD(\mathcal{Q}_o)$ over the formal log PD thickening $o\hookrightarrow \Spf W_{\bN^r}$ 
 . Through semi-stable comparison, one obtains a $\varphi$-invariant and monodromy stable ``CSpin Tate cycle'' $\bm{\pi}_{o}\in M_o^{\otimes (2,2)}$ which cuts out a reductive subgroup $G\subseteq \GSp(M_0,\lambda_o)$; cf. \cite[\S 3.2.2]{MP11}. More precisely, $\bm{\pi}_{o}$ is the image under semi-stable comparison of the $p$-adic étale tensor $\bm{\pi}_{p}\in \bH_{p,x}^{\otimes (2,2)}$ as per \S\ref{subsec: set up SV}, and $G$ identifies with $\GSpin(L_{W})$\footnote{Note that $G$ is \textit{a priori} only an inner form of $\GSpin(L_{W})$; cf. \cite[Corollary 2.2.4.3(4)]{MP11}. In our case, $W$ is strictly henselian and $G$ is smooth, so the $G$-torsor whose $S$-points are $\{\phi:\GSpin(L_{W})\otimes S\xrightarrow{\sim} G\otimes S| \phi(\bm{\pi}_{p})=\bm{\pi}_{o}\}$ admits a $W$-point. See also \cite[Corollary 1.3.6(3)]{KM09}.}. Following \cite[\S 3.2.3-3.2.5, \S3.3.1]{MP11}, one can use group theory to explicitly construct log $W$-algebras $R^+,R_{\sigma}, R^+_{G}, R_{\sigma_G}$, such that we have the following identifications: 
\begin{equation} 
\begin{tikzcd}
R \arrow[r, "\iota^\sharp"] \arrow[d,  Rightarrow,no head]                                      & R_G \arrow[d, Rightarrow,no head]     \\
R^+\widehat{\otimes}R_{\sigma} \arrow[r, "\iota^\sharp_+\widehat{\otimes}\iota^\sharp_{\sigma}"] &  R^+_{G}\widehat{\otimes} R_{\sigma_G}
\end{tikzcd}
\end{equation}
The vertical identifications follows essentially from the proof of \cite[Theorem 4.3.2.1]{MP11}. The map $\iota^\sharp_+:R^+\rightarrow R^+_G$ is a surjection, whereas $\iota^\sharp_{\sigma}:R_{\sigma} \rightarrow R_{\sigma_G}$ is the normalization of a surjection; cf. \cite[\S3.3.1]{MP11}. The ring $R^+$ carries a classical 1-motive $\mathcal{Q}^+$, while $R_{\sigma}$ carries a log 1-motive $\mathcal{Q}_{\sigma}$ with toroidal degree 0 piece. We can decompose $\cQ$ as the sum of the pullbacks of $\mathcal{Q}^+$ and $\mathcal{Q}_{\sigma}$, see \cite[\S 3.2.5]{MP11}. 
We can write down the Frobenius, filtrations, and log connection on $\bD(\mathcal{Q})$ using the explicit coordinates on $R$; cf. \cite[\S 3.2.6]{MP11}.  

Let $\cQ^+_G$ (resp. $\mathcal{Q}_{\sigma_G}$) be the base change of $\mathcal{Q}^+$ (resp. $\mathcal{Q}_\sigma$) to $R^+_G$ (resp. $R_{\sigma_G}$). Let $\cQ^{\mathrm{cl}}$ be the classical part of $\cQ$, and let $\cQ_G$ (resp. $\cQ_G^{\mathrm{cl}}$) be the base change of $\cQ$ (resp. $\cQ^{\mathrm{cl}}$) to $R_G$.   

\subsection{Special endomorphisms on the boundary}
To simplify the discussion, we will only consider the  special case where $o$ lies on a boundary stratum with $\sigma$ an one dimensional inner ray. Note that $r=1$ in this case. 

\subsubsection{Log K3-crystals} Note that $\bD(\cQ^{\mathrm{cl}})$ is a classical Dieudonné crystal over $R$. It can be obtained from $\bD(\cQ)$ by forgetting the log connection corresponding to $\bN^r\rightarrow R$.  In particular, if $\bm{s}$ is a tensor of $\bD(\cQ)^{\otimes}$, it canoincally gives rise to a tensor $\bm{s}^{\mathrm{cl}}$ of $\bD(\cQ^{\mathrm{cl}})^{\otimes}$.


\begin{defn}
Let $\bL\subseteq \End(\bD(\mathcal{Q}_G))$ be the image of the idempotent operator $\bm{\pi}_{G}$. Note that $\bL$ carries a Frobenius, a three step Hodge filtration $\Fil^\bullet \bL$ concentrated in degree $[-1,1]$, a five step weight filtration $W_\bullet \bL$ concentrated in degree $[-2,2]$ \footnote{The weight filtration is essentially 3 step: when $o$ is of type II then $\gr_{-2}=\gr_{2}=0$, and when $o$ is of type III, then $\gr_{-1}=\gr_{1}=0$.}, and a quadratic paring $Q_G$. We call $\bL$ a \textbf{log K3-crystal} over $R_G$. Similarly, by the process discussed above, the tensor $\bm{\pi}_{G}$ gives rise to a $\bm{\pi}_{G}^{\mathrm{cl}}\in \bD(\cQ_G^{\mathrm{cl}})^{\otimes (2,2)}$ whose image is a classical K3-crystal $\mathbb{L}^{\mathrm{cl}}\subseteq \End(\bD(\cQ_G^{\mathrm{cl}}))$. We will call $\mathbb{L}^{\mathrm{cl}}$ the \textbf{classical part} of $\mathbb{L}$.
\end{defn}
\subsubsection{Definition of special endomorphisms}
\begin{defn}\label{def:specialendoboundary} An element 
$s\in \End(\mathcal{Q}_o)$ is called a \textbf{special endomorphism}, if its crystalline realization lies in $W_0\mathbb{L}_o$ (hence further lies in  $(W_0\bL_o\cap \Fil^0\bL_o)^{\varphi=1})$. Let $T$ be a local log scheme, and let $ T\rightarrow \Spec R_G$ be a morphism of local log schemes. An element $s\in \End(\mathcal{Q}_T)$ is called a special endomorphism, if its restriction to the closed log point $o'\in T$ is the base change of a special endomorphism over $o$. The lattice of all special  endomorphisms of $\mathcal{Q}_T$ is denoted by $L(\mathcal{Q}_T)$.
\end{defn}
\begin{remark}
   We have $L(\mathcal{Q}_T)\subseteq L(\cQ_{o'})= L(\cQ_o)$. The proof of this fact relies on the assumption that the log structure on $R$ has chart $\bN\rightarrow R$, which guarantees that $\mathcal{Q}_{o'}$ does not admit more endomorphisms than $\cQ_o$ that come from relaxation of log structures. 
\end{remark}
\subsubsection{Deformation loci}
 Let $v\in L(\cQ_o)$. 
\begin{proposition}\label{prop:deforv}
There exists a $D_0(v)\in R_G$ not divisible by $p$, such that if $R_G\rightarrow R'$ is a map of complete local log rings, then $v$ extends to $\cQ\otimes R'$ if and only if $f(D_0(v))= 0 $. 
\end{proposition}
\begin{proof}
    Let $v^{\mathrm{cl}}\in \End(\cQ^{\mathrm{cl}}_o)$. By Lemma~\ref{lm:extlogII}, deforming $v$ is the same as deforming $v^{\mathrm{cl}}$, which lies in the realm of the deformation theory of classical 1-motives, and can be tackled by Serre--Tate theory; cf. \cite[Proposition 1.1.3.1]{MP11}. Let $\mathscr{G}=\cQ^{\mathrm{cl}}[p^\infty]$ be the associated $p$-divisible group. Then  $v^{\mathrm{cl}}$ deforms to $\cQ^{\mathrm{cl}}\otimes R'$ if and only if  $v^{\mathrm{cl}}[p^\infty]\in \End(\mathscr{G}_o)$ extends to $\mathscr{G}_{R'}$. Now the crystalline realization of $v^{\mathrm{cl}}$ lies in $\bL^{\mathrm{cl}}_o$. A similar argument as in \cite[Corollary 5.17]{MP15} shows that deformation locus of $v^{\mathrm{cl}}$ is cut out by a single element $D_0(v)\in R_G$. It remains to show that $p\nmid D_0(v)$. If it was not the case, then $v$ extends to all of $R_G\otimes k$. Let $\eta$ be the generic point of $R_G\otimes k$. Then the abelian variety over $\eta$ admits a special endomorphism, meaning that $\eta$ lies on a special divisor. This is absurd.    
\end{proof}

\begin{corollary}
    \label{prop:spcialendolatticeboundary}
Let $v\in \End(\cQ_o)$. Then 
$v\in L(\cQ_o)$ if and only if it deforms to a special endomorphism in the interior, i.e., there exists a DVR $D$ (with log structure given by the maximal ideal) with a map $\Spec D \rightarrow \mathscr{S}_W^\Sigma$ whose generic point $\eta$ maps to $\mathscr{S}_W$ while the special point $o'$ maps to $o$, such that the base change of $v$ to $o'$ is the specialization of a special endomorphism over $\eta$. In particular, $L(\cQ_o)$ is equipped with a positive definite pairing $\la -,-\ra$ characterized by $v\circ v=[\la v, v\ra]$.
\end{corollary}
\begin{proof}
Apply Proposition~\ref{prop:deforv}. Let $v\in L(\cQ_o)$, then there exists a $D=k[\![t]\!]$ with a map $\Spec D\rightarrow \Spec R_G/(D_0(v))$ whose generic point $\eta$ maps to the interior and whose closed point maps to $o$. This proves the only if part. The if part is easy. 
\end{proof}
\subsubsection{Weak moduli property of $\overline{Z(m)}$}
Suppose that $p\nmid m$. The completion $\overline{Z(m)}^{/o}$ is a union of irreducible formal divisors of $\Spf R_G\otimes k$. We denote the set of irreducible components by $\Pi$. Any $\mathscr{Z}\in \Pi$ is equipped with the obvious log structure from the boundary, and a special endomorphism (of abelian schemes) on the generic fiber. Thanks to \cite[Proposition 1.2.4.2]{MP11}, we get a special endomorphism (of log 1-motives) over the normalization $\widetilde{\mathscr{Z}}$. Specializing this endomorphism to $o$, we get an element $\theta_m(\mathscr{Z})\in L(\cQ_o)_m:=\{v\in L(\cQ_o)|\la v,v\ra =m\}$. This gives rise to a surjective map $\theta_m: \Pi\rightarrow L(\cQ_o)_m$.


\begin{proposition}\label{prop:moduliofM(m)}
For $v\in L(\cQ_o)_m$, let $\mathscr{D}_v\subseteq \Spf R_G\otimes k$ be the deformation locus. Then~\begin{enumerate}
    \item $\theta_m$ is a bijection, and  $\mathscr{Z}=\mathscr{D}_{\theta_m(\mathscr{Z})}$.
    \item Let $A$ be an Artinian local log ring with a map $\Spf A\rightarrow \Spf R_G\otimes k$. Then $v$ deforms to $\Spf A$ if and only if the map factors through $\mathscr{Z}=\theta_m^{-1}(v)$.   
\end{enumerate} 
\end{proposition}
\begin{proof}
    \begin{enumerate}
        \item Let $v=\theta_m(\mathscr{Z})$. The map $\widetilde{\mathscr{Z}}\rightarrow \Spf R_G$ factor through $\mathscr{D}_v$, because $v$ deforms to $ \widetilde{\mathscr{Z}}$. This shows that $\mathscr{Z}\subseteq \mathscr{D}_v$. Since both are divisors and $\mathscr{D}_v$ is irreducible, $\mathscr{Z}= \mathscr{D}_v$. To prove the injectivity of $\theta_m$, suppose that $\mathscr{Z}_1$ and $\mathscr{Z}_2$ are two different elements in $\theta_m^{-1}(v)$. Consider a DVR $D/k$ with a map $\Spf D\rightarrow \mathscr{D}_v$ 
        whose (rigid) generic point lies in the interior. Then it gives rise to two maps $\Spf D\rightarrow \mathscr{Z}_i$, $i=1,2$, whose (rigid) generic point maps to two different points of $Z(m)$. As a result, the (schematic) generic point of $D$ has two different special endomorphisms specializing to $v$; cf. Lemma~\ref{eq:existencetheorem}, contradicting the unique lifting property.
        \item This follows from the previous part.  \end{enumerate}
\end{proof}
\begin{proposition}
    \label{corL}
In Setup~\ref{setp:logmorphi}, we have, for $p\nmid m$,$$l_P(m)=\sum_{n\geq 1}\#\{v\in L(\cQ_{k[\![t]\!]/(t^n)})|\la v,v \ra=m\}.$$ 
\end{proposition}
\begin{proof}
  It follows from Proposition~\ref{prop:moduliofM(m)}.
\end{proof}

\subsection{Proof of uniform decay} We prove Theorem~\ref{thm:globalanalysis}, assuming that $k=\Fpbar$.
Denote by $\eta$  the generic point of $\Spec k[\![t]\!]$. 
    Let $\Lambda \subset L(\cQ_o) \otimes \bZ_p$ be the module of formal special endomorphisms that do not decay. Let $\mathscr{G}$ be the log $p$-divisible group over $\Spec k[\![t]\!]$ associated to the pullback of $\cQ$. By Lemma~\ref{eq:existencetheorem}, each element  of $\Lambda$ extends to an endomorphism of $\mathscr{G}$ and, by restricting to the generic fiber, gives rise to a formal special endomorphism of the $p$-divisible group over $\eta$. By Theorem \ref{theorem: algebraic}, we have that $\Lambda$ is induced by a special subvariety of our ambient Shimura variety. All special subvarieties are associated to groups of the form $\Res_{F/\bQ} H$, where $F$ is a totally real field, and $H$ is an orthogonal group (resp. a unitary group) having signature $(a,2)$ at one real place and compact at all other places (resp. having signature $(a,1)$ at one real place and compact at all other real places). Such a Shimura variety can only meet the boundary if $F = \bQ$. By assumption, $C$ is not contained in any special divisor. Hence, the only possibility is that $C$ is contained in a $\mathrm{U}(1,n)$ Shimura subvariety associated to an imaginary quadratic field. In this case, global considerations yield that $\Lambda$ is a totally isotropic subspace of $L(\cQ_o)\otimes \bZ_p$. 

    Note that there is some $N$ such that the module of special endomorphisms mod $t^N$ is contained in $\Lambda + pL(\cQ_o)\otimes \bZ_p$ (otherwise, for each $n$ the module of special endomorphisms mod $t^n$ contains an element $a_n\notin \Lambda + pL(\cQ_o)\otimes \bZ_p$, then any limit point of $\{a_n\}$ needs to lie in $\Lambda$, while staying away from $\Lambda + pL(\cQ_o)\otimes \bZ_p$, which is a contradiction). 
  On the other hand, every $w\in \Lambda + pL(\cQ_o)\otimes \bZ_p$ has the property $p\mid \la w,w\ra$ -- this follows directly from the fact that $\Lambda$ is an isotropic subspace. The result now follows. 
  
 $\hfill\square$


\section{Diophantine approximation for special endomorphisms}\label{sec:dioapp}
In this section we estimate the tail $l_P^{\geq \delta}(m)$. The key input is Theorem~\ref{thm:SubspaceEndo}, which asserts that integral points in a quadratic lattice can not $p$-adically well approximate a strongly irrational subspace. The proof of Theorem~\ref{thm:SubspaceEndo} relies on a non-archimedean generalization of Schmidt's subspace theorem, as well as classical results on quadratic lattices. We do not assume that $p\nmid m$.

\begin{theorem}\label{thm:tailestimate}
Setup be as in \ref{subsub:maintech}. Let $P$ be an interior point on $C$ and let $\Lambda\subseteq L_{P,1}\otimes \bZ_p$ be the lattice of vectors that do not decay. Let $d=\rk L_{P,1}\geq 2$, $r=\rk \Lambda$ and $\mu=1-\frac{2(d-2)}{r+1}$. Let $\delta\geq 0$ be a real number. Then for any $\epsilon>0$, we have  $$l_P^{\geq \delta}(m)=O\left(m^{\frac{d}{2}-1+\max\{\frac{\delta}{2}\mu,\frac{r+1}{2}\mu\}+\epsilon}\right),$$
     where the implied constant does not depend on $\epsilon$.
\end{theorem}
The proof of this theorem relies on the following result which we establish first.

\begin{theorem}\label{thm:SubspaceEndo}Let $(L,Q)$ be a positive definite quadratic $\bZ$-lattice of rank $d\geq 2$. Let $p$ be a prime. Suppose that 
$\mathfrak{V}\subsetneq L_{\overline{\bQ}}$ is a $\overline{\bQ}$-subspace of dimension $r\geq 0$ that contains no rational ray. For $n\in \bN$, let $L\la n \ra:=L\cap (\mathfrak{V}+p^nL_{\overline{\bQ}})$. Let $c=\frac{r+1}{2}$. Then the following are true: 
\begin{enumerate}
    \item 
$\forall \epsilon\in \bR^+\forall m\in \bN^+$, if $n>(c+\epsilon)\log_p m$, then $L\la n \ra$ contains no element with $Q(v)=m$. 
\item $\min_{v\in L\la n \ra} Q(v)\geq p^{\frac{n}{c}}$.
\item there is a constant $c'$ depending only on $d$, such that $\forall n\in \bN \forall m\in \bN^+\forall \epsilon\in \bR^+$,
$$\#\{v\in L\la n \ra:Q(v)= m\}\leq c'm^{\frac{1}{2}(d-2)+\epsilon}p^{-\frac{n}{c}(d-2)}.$$
\end{enumerate}
\end{theorem}
The theorem relies on a nonarchimedean generalization of the subspace theorem. 

\begin{lemma}[$p$-adic subspace theorem]\label{lm:subspacethm}
Let $d>0$, and let $l_1,l_2,...,l_d\in \overline{\bQ}[X_1,X_2,...,X_d]$ be linearly independent linear forms. Let $Q$ be a positive definite quadratic pairing on $\bZ^d$. Let $p$ be a prime. Then for any $\epsilon>0$, the solutions $\mathbf{x}=(x_1,x_2,...,x_d)\in \bZ^d$ of the inequality 
$$\prod_{i=1}^d |l_{i}(\mathbf{x})|_p\leq Q(\mathbf{x})^{-\frac{d}{2}-\epsilon} $$
 is contained in finitely many hyperplanes of $\bQ^d$. Here $|\cdot|_p$ is the standard $p$-adic absolute value on $\bQ_p$ extended to $\overline{\bQ}_p$. 
\end{lemma}
\begin{proof}
    Let $|\!|\mathbf{x}|\!|=\max_{1\leq i\leq d}\{|x_i|\}$, here $|\cdot| $ is the archimedean absolute value. By Schlickewei's nonarchimedean generalization of subspace theorem; cf. \cite{Schlickewei},  
    the solution to the inequality 
    $$\prod_{i=1}^d |x_i|\cdot |l_{i}(\mathbf{x})|_p\leq |\!|\mathbf{x}|\!|^{-\epsilon}$$
  is contained in finitely many hyperplanes of $\bQ^d$. The lemma then follows by the existence of $\alpha_2>\alpha_1>0$ such that $\alpha_1 |\!|\mathbf{x}|\!|\leq Q(\mathbf{x})^{\frac{1}{2}}\leq \alpha_2 |\!|\mathbf{x}|\!|$. 
\end{proof}

\begin{proof}[Proof of Theorem~\ref{thm:SubspaceEndo}] 
 
\begin{enumerate}
    \item Let $\epsilon>0$. 
 Fix an identification $L=\bZ^d$. We pick $l_1\sim l_d$ in Lemma~\ref{lm:subspacethm} such that $l_{r+1},...,l_{d}$ generates the ideal that cuts out $\mathfrak{V}$. Suppose that $n>(c+\epsilon)\log_p m$ and $\mathbf{x}\in L\cap (\mathfrak{V}+p^nL_{\overline{\bQ}})$ with $Q(\mathbf{x})=m$, then 
$$\prod_{i=1}^d |l_{i}(\mathbf{x})|_p\leq p^{-n(d-r)}<m^{-(c+\epsilon)(d-r)}\leq Q(\mathbf{x})^{-\frac{d}{2}-(d-r)\epsilon},$$
where the last inequality follows from the fact that $c\geq \frac{d}{2(d-r)}$ by Lemma~\ref{lm:easyone}. Therefore the lemma implies that such $\mathbf{x}$ need to lie in finitely many hyperplanes of $L_{\bQ}$. Let $L'_{\bQ}$ be such a hyperplane, and let $\mathfrak{V}'=\mathfrak{V}\cap L'_{\overline{\bQ}}$. We run the same argument to show that all such $\mathbf{x}$ lying in $L'_{{\bQ}}$ need to furthermore lie in finite many hyperplanes in $L'_{{\bQ}}$ (note that by Lemma~\ref{lm:easyone} we again have $c\geq \frac{d'}{2(d'-r')}$, where  $d'= \dim L'_{\bQ}$ and $r'=\dim \mathfrak{V}'$). We then do downward induction on the dimension and use Lemma~\ref{lm:easyone} at each step. It finally turns out that all such $\mathbf{x}$ are contained in finitely many lines, from which it is clear that no $\mathbf{x}$ exists.
\item Follows directly from (1). 
\item Write $Q$ as a $d\times d$ symmetric matrix $M$ under an integral basis of $L\la n \ra$. From \cite[\S 3]{Waibel21}, we can choose a suitable basis so that $M$ lies in the Siegel domain $\cS(\frac{4}{3},\frac{1}{2})$; this means that there exists a decomposition $M=V^tDV$ over $\bR$ such that $V$ is an upper-triangular matrix with diagonal entries $1$ and $|v_{ij}|\leq \frac{1}{2}$, and $D=\diag(a_1,a_2,...,a_d)$ with $0<a_i\leq \frac{4}{3}a_{i+1}$. Let
$\mu_1\leq \mu_2\leq ...\leq \mu_d$ be the successive minima for $L\la n \ra$; cf. \cite[Definition 2.2]{Eskin}. Then there is a constant $c_0$ only depending on $d$ such that $a_i\geq c_0\mu_i^2$ for $1\leq i\leq d$. We further notice that $\mu_1^2\geq p^{\frac{n}{c}}$ by (2). So $a_i\geq c_0p^{\frac{n}{c}}$ for all $i$. Mimicking the proof of \cite[Lemma 9]{Waibel21}, we can write 
$$Q(\mathbf{x})=\frac{1}{2}\mathbf{x}^{t}M\mathbf{x}=\frac{a_1}{2}(x_1+v_{12}x_2+\cdots+v_{1d}x_d)^2+\cdots+\frac{a_d}{2}x_d^2.$$
To count the solutions for $Q(\mathbf{x})=m$, for $x_3,x_4,...,x_m$ there are at most 
$$c_1\left(2\sqrt{\frac{m}{a_3}}+1 \right)\left(2\sqrt{\frac{m}{a_4}}+1 \right)\cdots\left(2\sqrt{\frac{m}{a_d}}+1 \right)$$
choices for some constant $c_1$ only depending on $d$. After that, we are left with a quadratic form in $x_1,x_2$ that has no more than $c_3 n^{\epsilon}$ solutions for an absolute constant $c_3$, and any $\epsilon$. We obtain the desired bound by combining above estimates.
\end{enumerate}

\end{proof}
\begin{lemma}\label{lm:easyone}
    Let the setup be as in Theorem~\ref{thm:SubspaceEndo}. Let $\Pi\subseteq L_{\bQ}$ be a $\bQ$-subspace. Then  $$\frac{\dim \Pi}{2(\dim \Pi - \dim \mathfrak{V}\cap \Pi_{\overline{\bQ}})}\leq c.$$
\end{lemma}
\begin{proof}
Let $e=\dim \Pi$ and $f=\dim \mathfrak{V}\cap \Pi_{\overline{\bQ}}$. Note that $e>f$ since $\mathfrak{V}$ does not contain any rational ray. It follows that $r(e-f) \geq r\geq f$. This implies that $r+1\geq \frac{e}{e-f}$, and we are done. 
\end{proof}

\begin{proof}[Proof of Theorem~\ref{thm:tailestimate}]
By Theorem~\ref{thm: Qbar algebraicity of formal special endomorphisms}, there is a $\overline{\bQ}$-subspace $\mathfrak{V}\subseteq L_{P,1}\otimes \overline{\bQ}$ containing no rational ray  
such that $\mathfrak{V}\otimes \overline{\bQ}_p=\Lambda\otimes \overline{\bQ}_p$. We  apply Theorem~\ref{thm:SubspaceEndo} to the lattice $(L_{P,1},Q)$. Let $L_{P}\la n \ra:=L_{P,1}\cap (\mathfrak{V}+p^nL_{P,1}\otimes \overline{\bQ})$. By Corollary~\ref{cor:control_subspace}, there is a positive integer $D$ such that for all $n\geq 0$, we have \begin{equation}\label{eq:inclusion1}
    L_{P,(1+p+p^2+...+p^{n})D+1}\subseteq L_{P}\la n+1\ra.
\end{equation} 
 Let  $\delta\geq 0$, then for $m\gg 0$, we have $(1+p+p^2+...+p^{[\frac{\delta}{2}\log_p m]-1})D+1\leq m^\delta$. So (\ref{eq:inclusion1}) implies that for $m\gg 0$,  
\begin{align*}
   l_P^{\geq\delta}(m) 
\leq \sum_{n\geq [\frac{\delta}{2}\log_p m]}Dp^{n}\#\{v\in L_{P}\la n \ra:Q(v)=m\}.
\end{align*}
    Applying Theorem~\ref{thm:SubspaceEndo}(1)(3), we see that for any $\epsilon >0$, 
$$l_P^{\geq\delta}(m) \leq D c'm^{\frac{d-2}{2}+\frac{\epsilon}{2}}\sum^{(c+\frac{\epsilon}{2})\log_p m}_{n= \lfloor\frac{\delta}{2}\log_p m\rfloor}p^{n(1-\frac{d-2}{c})}=O\left(m^{\frac{d}{2}-1+\max\{\frac{\delta}{2}\mu,c\mu\}+\epsilon}\right).$$
\end{proof}

\section{Proof of Theorem~\ref{thm: main version 2}}\label{sec:pvofmain}

\subsection*{Supersingular case} As noted in \S\ref{sub: intersection numbers}, to prove Theorem~\ref{thm: main version 2}(1) it suffices to show Lemma~\ref{tail}. Let $\delta>0$, since  $l_P^{\geq\delta}(m)$ only gets smaller when $\delta$ gets larger, we can assume that $\delta\in (0,{r+1})$. 
Apply Theorem~\ref{thm:tailestimate}, and note that we have $d=b+2$ and $d\geq r+3$. This implies that $\mu<0$, so $\max\{\frac{\delta}{2}\mu,\frac{r+1}{2}\mu\}=\frac{\delta}{2}\mu$. Taking $\epsilon=-\frac{\delta}{4}\mu$, we find that $l_P^{\geq\delta}(m)=O(m^{\frac{b}{2}+\frac{\delta}{2}\mu-\frac{\delta}{4}\mu})=o(m^{\frac{b}{2}})$. This proves Lemma~\ref{tail}, hence Theorem~\ref{thm: main version 2}(1).

\subsection*{Nonsupersingular case} We will leave to the reader the  trivial case where $\rk L_{P,1}=1$. Now we assume $d\geq 2$ and apply Theorem~\ref{thm: main version 2}(1) to $\delta=0$ and note that we have $b\geq d$. If $\mu<0$, then we $l_P(m)=l_P^{\geq 0}(m)=O(m^{\frac{d}{2}-1+\epsilon})=o(m^{\frac{b}{2}})$. If $\mu\geq 0$, then $\frac{d}{2}-1+\max\{\frac{\delta}{2}\mu,\frac{r+1}{2}\mu\}= \frac{r+1}{2}-\frac{d}{2}+1$. If $ r+1<2(d-1)$, then  $l_P(m)=o(m^{\frac{b}{2}})$ by easy computation. If $ r+1\geq 2(d-1)$, we must have $(r,d)=(1,2)$. Then the lattice $\Lambda$ is of rank 1, meaning that $C$ lies on special divisor by Theorem~\ref{theorem: algebraic}, which can not happen. This concludes Theorem~\ref{thm: main version 2}(2). 

\subsection*{Boundary case} 
To show that $C\cdot \overline{Z(m)}=o(m^{\frac{b}{2}})$, we apply Proposition~\ref{corL}. By the uniform decay Theorem~\ref{thm:globalanalysis},
it suffices to show that $\#\{v\in L(\cQ_o):\la v,v\ra=m\}=o(m^{\frac{b}{2}})$. This follows simply from the well-known bound on the coefficients of the theta series associated to the quadratic lattice $L(\cQ_o)$, as well as the fact that $\rk L(\cQ_o)\leq \rk (W_0\bL_o\cap \Fil^0\bL_o)^{\varphi=1} \leq b-1$. This concludes Theorem~\ref{thm: main version 2}(3). $\hfill\square$
\bibliographystyle{alpha}
\bibliography{ref}

\noindent Ruofan Jiang {\footnotesize\,\,\, University of California, Berkeley, Department of Mathematics, Evans Hall,
Berkeley, CA 94720-3840, USA \,\,\,  Email: \url{ruofanjiang@berkeley.edu}}\\

\noindent Ananth Shankar {\footnotesize\,\,\, Northwestern University, Department of Mathematics, Lunt Hall,
Evanston, Il 60208, USA \,\,\,  Email: \url{ananth@northwestern.edu}}
\end{document}